\documentclass[12 pt]{article}
\usepackage{graphicx}
\usepackage{mathtools,amsmath,amssymb,amsthm,enumerate,hyperref, enumerate, xcolor}
\usepackage{mathrsfs}

\newtheorem{theorem}{Theorem}[section]
\newtheorem{corollary}[theorem]{Corollary}
\newtheorem{lemma}[theorem]{Lemma}
\newtheorem{proposition}[theorem]{Proposition}

\newtheorem{letterthm}{Theorem}
\newtheorem{lettercor}[letterthm]{Corollary}

\theoremstyle{definition}
\newtheorem{definition}[theorem]{Definition}
\newtheorem{example}[theorem]{Example}

\theoremstyle{remark}
\newtheorem{remark}[theorem]{Remark}

\newcommand{\cA}{\mathcal{A}}
\newcommand{\cB}{\mathcal{B}}
\newcommand{\cD}{\mathcal{D}}
\newcommand{\cE}{\mathcal{E}}
\newcommand{\cF}{\mathcal{F}}
\newcommand{\cG}{\mathcal{G}}

\newcommand{\cR}{\mathcal{R}}
\newcommand{\cS}{\mathcal{S}}

\newcommand{\cZ}{\mathcal{Z}}

\newcommand{\C}{\mathbb{C}}
\newcommand{\F}{\mathbb{F}}
\newcommand{\N}{\mathbb{N}}

\newcommand{\T}{\mathbb{T}}
\newcommand{\Z}{\mathbb{Z}}

\newcommand{\scF}{\mathscr{F}}

\newcommand{\vphi}{\varphi}
\newcommand{\actson}{\curvearrowright}

\newcommand{\inv}{^{-1}}
\newcommand{\zero}{^{(0)}}
\newcommand{\two}{^{(2)}}
\newcommand{\three}{^{(3)}}
\newcommand{\ovt}{\overline{\otimes}}

\DeclareMathOperator{\id}{id}

\DeclareMathOperator{\Fix}{Fix}
\DeclareMathOperator{\Iso}{Iso}
\DeclareMathOperator{\Span}{span}
\DeclareMathOperator{\supp}{supp}

\DeclareMathOperator{\LC}{LC}
\DeclareMathOperator{\RC}{RC}

\makeatletter
\let\@fnsymbol\@alph
\makeatother

\title{Factoriality of groupoid von Neumann algebras}
\author{Tey Berendschot \thanks{KU Leuven, Department of Mathematics, Leuven, Belgium} \textsuperscript{,}\thanks{\texttt{tey.berendschot@gmail.com}\\ Supported by PhD grant 1101324N funded by the Research Foundation Flanders (FWO)} \and Soham Chakraborty \footnotemark[1] \textsuperscript{,}\thanks{\texttt{soham.chakraborty@kuleuven.be} \\ Supported by FWO research project G090420N of the Research Foundation Flanders} \and Milan Donvil \footnotemark[1] \textsuperscript{,}\thanks{\texttt{milan.donvil@kuleuven.be}\\ Supported by PhD grant 1162024N funded by the Research Foundation Flanders (FWO)} \and Se-Jin Kim \footnotemark[1] \textsuperscript{,}\thanks{\texttt{sam.kim@kuleuven.be} \\ Supported by the Research Foundation Flanders (FWO) project number G085020N and the internal KU Leuven funds project number C14/19/088.}  }

\begin{document}

\setlength{\parindent}{0em}
\setlength{\parskip  }{5.5 pt}

\maketitle

\begin{abstract}\noindent
We give a characterisation of the factoriality of the groupoid von Neumann algebra $L(\cG)$ associated to a discrete measured groupoid $(\cG, \mu)$. We introduce the notion of groupoids with `infinite conjugacy classes' and show that this property together with ergodicity of the groupoid is equivalent with factoriality of $L(\cG)$.
\end{abstract}

\section{Introduction}

Discrete measured groupoids have historically played an important role in the study of von Neumann algebras. In \cite{Mac63}, Mackey introduced the notion of ergodic groupoids. An early breakthrough in this direction was due to Feldman and Moore who showed in \cite{FM77} that every principal discrete measured groupoid, i.e. countable equivalence relation, arises as the orbit equivalence relation of a countable group action. In \cite{FM77b} they defined the von Neumann algebra $L(\mathcal{R})$ associated to such an equivalence relation $\mathcal{R}$ on a standard probability space $(X,\mu)$. They showed that the abelian subalgebra $L^{\infty}(X,\mu)$ is a so-called Cartan subalgebra of $L(\mathcal{R})$. In fact they showed that any von Neumann algebra with a Cartan subalgebra arises from such an equivalence relation together with a 2-cocycle, thus giving a complete characterization. 

In general, given a discrete measured groupoid $\mathcal{G}$, one can similarly construct the groupoid von Neumann algebra $L(\mathcal{G})$. This was first documented and studied by Hahn in \cite{Hah78}. Groupoid von Neumann algebras provide a huge class of examples that cover group von Neumann algebras, von Neumann algebras with Cartan subalgebras as above and more generally, group measure space von Neumann algebras, arising from transformation groupoids. This leads to an interesting question: can every factor be realized as a groupoid von Neumann algebra? It can be shown that such a groupoid von Neumann algebra is anti-isomorphic to itself. In \cite{Con75}, Connes constructed a factor which is not anti isomorphic to itself, thus answering the question in negative. 

In this article, we study the question: when is the von Neumann algebra associated to a discrete measured groupoid a factor? We give a characterization of factoriality purely in terms of the groupoid. Up until this point, such a characterization was only known in special cases. For a countable group $\Gamma$, it is well known that $L(\Gamma)$ is a factor if and only if the group has infinite conjugacy classes (such groups are generally called icc). In the case of a countable equivalence relation $\cR$ it was proved in \cite{FM77b} that the factoriality of the von Neumann algebra $L(\cR)$ is equivalent to the ergodicity of $\cR$. Finally, when $\Gamma \curvearrowright (X,\mu)$ is an ergodic (not necessarily free) action of a countable group on a standard probability space, a characterization of the factoriality of the crossed product $L^\infty(X, \mu) \rtimes \Gamma$ was given by Vaes in \cite{362329}.

The organisation of the first part of the paper is as follows. We state some preliminaries on discrete measured groupoids and their von Neumann algebras in Section \ref{Section:Preliminaries}, and then gather some folklore results on bases of groupoids in Section \ref{Section:Basis, Fourier, j map}. A lot of these results have already been used intrinsically in the literature, for example in \cite{Anantharaman-Delaroche13}. In Section \ref{Section:Factoriality untwisted}, we define the notion of conjugacy classes for Borel subsets of the isotropy subgroupoid, which consists of all elements whose source and target are equal. For such a Borel subset $E$ we show that the set $\Omega_{E}$ consisting of all conjugates of $E$ is a Borel subset and we call such a conjugacy class finite if it has finite measure. In Definition \ref{Definition:Finite measure conjugacy class and icc} we define groupoids satisfying the \emph{infinite conjugacy class (icc) condition} if any such subset with a finite conjugacy class is essentially a subset of the unit space. Then we prove the following theorem, which is the main result of this paper.

\begin{letterthm}\label{theorem:Factoriality untwisted}
    Let $(\cG, \mu)$ be a discrete measured groupoid. The following two conditions are equivalent:
    \begin{enumerate}
        \item The groupoid $(\cG,\mu)$ is icc.
        \item We have that $\cZ \left(L(\cG)\right) = L^\infty(\cG\zero, \mu)^\cG$.
    \end{enumerate}
\end{letterthm}

When the groupoid is furthermore ergodic, we have the following immediate characterization of factoriality of groupoid von Neumann algebras from Theorem \ref{theorem:Factoriality untwisted}. 

\begin{lettercor}\label{Corollary:Factoriality untwisted}
    Let $(\cG, \mu)$ be a discrete measured groupoid. The following two conditions are equivalent:
    \begin{enumerate}
        \item The groupoid $(\cG,\mu)$ is ergodic and icc.
        \item The von Neumann algebra $L(\cG)$ is a factor.
    \end{enumerate}
\end{lettercor}

The situation is more complicated when a non-trivial cocycle twist is involved. Indeed for a 2-cocycle $\omega$ on $\cG$, one can define the twisted groupoid von Neumann algebra as we do in Section \ref{Section:Factoriality twisted}. However, a lot of statements for such twisted von Neumann algebras become more challenging become more challenging to prove. For example, as observed in \cite{Hah78}, the result in \cite{Con75} does not address the question: `Can every factor can be realized as a cocycle twisted groupoid von Neumann algebra?' which remained open for a long time. In fact, only recently in \cite{donvil2024wsuperrigidity}, the third named author and Stefaan Vaes constructed a family of II$_{1}$ factors that are not virtually isomorphic to any twisted groupoid von Neumann algebra, even allowing arbitrary amplifications. 

Similarly for factoriality, beyond the principal case the 2-cocycle creates a non-trivial obstruction. A characterization of factoriality for twisted group von Neumann algebras  was given in \cite{Kle62}. In Section \ref{Section: Kleppner's condition and Factoriality}, we introduce the notion of central subsets of the isotropy subgroupoid. If $\omega$ is a 2-cocycle on the groupoid, we call such a conjugation invariant Borel subset $E \subset \Iso(\cG)$ \textit{central} if there exists a function $f: E \rightarrow \mathbb{C}$ such that $f(ghg^{-1}) = \overline{\omega(ghg\inv, g)}\omega(g, h) f(h)$ for elements $h \in E$ and all elements $g$ with the same source as $h$. In Theorem \ref{theorem:Factoriality twisted} and Corollary \ref{corollary:factoriality twisted}, we give a characterization of factoriality for cocycle twisted groupoid von Neumann algebras in terms of these central subsets. We also introduce a similar condition as in \cite{Kle62} for twisted discrete measure groupoids, and show that it is a necessary condition for the factoriality of the twisted groupoid von Neumann algebra. 

Finally in Section \ref{Section:Corollaries and examples} we give some applications of our main theorem. Firstly, we show that groupoids arising as a bundle of groups are icc if and only if the isotropy groups are icc almost everywhere. Next we reprove the theorem of Vaes for transformation groupoids and generalize his results in the context of partial dynamical systems. The paper ends with an analysis of measurable Deaconu--Renault groupoids.

\section{Preliminaries}\label{Section:Preliminaries}

A \emph{discrete Borel groupoid} is a groupoid $\cG$ that is also a standard Borel space such that
\begin{itemize}
	\item The unit space $\cG\zero$ is a Borel subset of $\cG$.
	\item the source and target maps $s,t: \cG \to \cG^{(0)}$ are countable-to-one Borel maps,
	\item the composition map
\begin{align*}
     \{(g,h) \in \cG \times \cG \mid t(h) = s(g)\} \to \cG : (g,h) \mapsto g h 
 \end{align*}
and the inversion map $g \in \cG \mapsto g^{-1}$ are Borel maps.
\end{itemize}
See \cite[Section 2.1]{Anantharaman-Delaroche13} and the references therein for more information about Borel groupoids. A \emph{(non-singular) discrete measured groupoid} $(\cG, \mu)$ is a discrete Borel groupoid $\cG$ together with a probability measure $\mu$ on $\cG\zero$ such that the following two $\sigma$-finite measures on $\cG$ are equivalent:
\begin{align*}
\mu_s(A) &= \int_{\cG\zero} | \{g \in A\ |\ s(g) = a \}|\ \mathrm{d}\mu(a);\\
\quad \mu_t(A) &= \int_{\cG\zero} | \{g \in A\ |\ t(g) = a \}|\ \mathrm{d}\mu(a).
\end{align*}
In particular, if the two measures coincide, i.e. $\mu_t = \mu_s$, we say $(\cG, \mu)$ is a \emph{discrete probability measure preserving (pmp) groupoid}, and in that case we write $\mu^{(1)}$ for $\mu_t = \mu_s$.

Associated to a discrete Borel groupoid $\cG$ is an equivalence relation on $\cG\zero$ defined as $\cR_{\cG}=\{(t(g),s(g))\mid g\in \cG\}$. It is a nonsingular countable equivalence relation precisely because $\cG$ is a nonsingular discrete measured groupoid. We say $\cG$ is an \emph{ergodic} groupoid if its associated equivalence relation is ergodic.

The isotropy subgroupoid, denoted by $\Iso(\cG)$, consists of all elements that have the same source and range, i.e.
\begin{equation*}
    \Iso(\cG) = \{g \in \cG\ |\ s(g) = t(g) \}.
\end{equation*}

We call a subset $A \subset \cG$ a \emph{(partial) bisection} if the maps $s|_A$ and $t|_A$ are injective. If moreover we have that $s(A) = t(A) = \cG\zero$, we call $A$ a \emph{global bisection}. The set of all partial bisections of $\cG$ is called the \emph{full pseduogroup} of $\cG$ and denoted by $[[\cG]]$. The group of all global bisections is called the \emph{full group} of $\cG$ and denoted by $[\cG]$. For any two subsets $A, B \subset \cG$, denote by $A \cdot B = \{gh\ |\ g\in A, h \in B, s(g) = t(h)\}$. Note that if $A$ and $B$ are bisections, then so is $A \cdot B$.

For a bisection $A \subset \cG$, we denote by $\sigma_A := t \circ s|_A\inv$ the bijection from $s(A)$ to $t(B)$. If $(\cG, \mu)$ is a discrete measured groupoid, this induces a $\ast$-endomorphism on $L^\infty(\cG\zero, \mu)$ by
\begin{equation*}
    \sigma_A(f) (x) = \begin{cases}
        f \circ \sigma_A\inv & x \in t(A);\\
        0 & x \notin t(A).
    \end{cases}
\end{equation*}

Let $(\cG, \mu)$ be a discrete measured groupoid. To construct the groupoid von Neumann algebra, we start by defining partial isometries $\lambda_A \in B(L^2(\cG, \mu_s))$ for each bisection $A \subset \cG$ by
\begin{equation}\label{equation:Left regular rep}
    \lambda_A 1_E :=  1_{A\cdot E}
\end{equation}
for every subset $E \subset \cG$ with $\mu_s(E) < \infty$.
The \emph{groupoid von Neumann algebra} is defined by $L(\cG) := \{ \lambda_A\ |\ A \subset \cG \text{ Borel bisection}\}''$. See \cite{Anantharaman-Delaroche13} for more information on the construction. Whenever $A$ and $B$ are bisections, one sees that $\lambda_A \lambda_B = \lambda_{A \cdot B}$. Note that whenever $A \subset \cG$ is a Borel subset with $\mu_s(A) < \infty$, it defines an element of $L(\cG)$ which we also denote by $\lambda_A$ as it is defined by the formula \eqref{equation:Left regular rep}.

Similarly, on can define partial isometries by
\begin{equation*}
    (\rho_A 1_E) (g) := 1_{E\cdot A\inv}
\end{equation*}
and define the von Neumann algebra $R(\cG) := \{ \rho_A\ |\ A \subset \cG \text{ is a bisection}\}''$. It is clear that $\lambda_A$ and $\rho_B$ commute for all bisections $A, B \subset \cG$. In Proposition~\ref{Proposition:Commutant of left regular representation} we show that, as in the case of groups, $R(\cG)' = L(\cG)$. 

Note that $L(\cG)$ includes a copy of $L^\infty(\cG\zero, \mu)$ as a regular subalgebra. When $A \subset \cG$ is a Borel bisection and $F \in L^\infty(\cG\zero,\mu)$, one has the relation $\lambda_A F \lambda_A^* = (F \circ \sigma_A\inv) \cdot 1_{t(A)}$. Moreover, there is always a faithful normal conditional expectation $E: L(\cG) \to L^\infty(\cG\zero, \mu)$ given by $a \mapsto a1_{\cG\zero}|_{\cG\zero}$. The fact that $E(a) \in L^\infty(\cG\zero, \mu)$ follows from Lemma \ref{Lemma:a^B and a_B}. 

\section{Basis of a groupoid, Fourier decomposition, and the $j$ map}\label{Section:Basis, Fourier, j map}

The following proposition is a consequence a theorem of Lusin-Novikov, see e.g. \cite[Theorem~18.10]{Kec95}, as noted by Anatharaman-Delaroche in \cite{Anantharaman-Delaroche13}. We give a proof for convenience.

\begin{proposition}\label{Proposition:basis of groupoid}
    Let $\cG$ be a discrete Borel groupoid. There exists a countable collection $\cB$ of Borel subsets of $\cG$ such that the following hold.
    \begin{enumerate}
        \item $\cG$ is equal to the disjoint union $\cG = \bigsqcup_{B\in \cB}B$.
        \item Every $B \in \cB$ is a bisection.
        \item For every $B \in \cB$ we have that $B\inv \in \cB$.
        \item $\cG\zero \in \cB$.
    \end{enumerate}
\end{proposition}
\begin{proof}
    Since $s: \cG \to \cG\zero$ is countable-to-one, we can use the Lusin-Novikov Theorem \cite[Theorem~18.10]{Kec95} to find a countable collection $(U_n)_{n}$ of Borel sets $U_n \subset \cG$ such that $s: U_n \to s(U_n)$ is a Borel isomorphism for every $n$ and such that $\cG = \bigsqcup_n U_n$. Indeed, setting $P := \{(x,y) \in \cG\zero \times \cG\ |\ x = s(y) \}$, by Lusin-Novikov there is a partition $P = \bigsqcup_n P_n$. Denoting by $\pi$ the projection from $\cG\zero \times \cG$ onto $\cG$, set $U_n = \pi(P_n)$. Similarly, there exists a countable collection $(V_n)_n$ of Borel subsets $V_n \subset \cG$ such that $t: V_n \to t(V_n)$ is a Borel isomorphism for every $n$ and such that $\cG = \bigsqcup_n V_n$. Define $\cB' = \{(U_n \cap V_m) \setminus \cG_0 \}_{n,m}$ and set $\cB = \{A \cap B\inv\ |\ A, B \in \cB'\} \cup \{\cG\zero\}$.
\end{proof}

\begin{definition}
    Let $\cG$ be a discrete Borel groupoid. A countable collection $\cB$ of mutually disjoint Borel bisections of $\cG$ is a \emph{basis} of $\cG$ if $\cG\zero \in \cB$ and $\cup_{B\in \cB} B = \cG$. It is called a \emph{symmetric basis} if it moreover holds that whenever $B \in \cB$, also $B\inv \in \cB$.
\end{definition}

Let $(\cG, \mu)$ be a discrete measured groupoid. Our goal is to use a basis $\cB$ of $\cG$ to get a Fourier decomposition for elements of $L^2(\cG, \mu_s)$ that in turn yields a Fourier decomposition for elements of $L(\cG)$.

\begin{proposition}\label{Proposition:basis for L2}
    Let $(\cG, \mu)$ be a discrete measured groupoid and let $\cB$ be a basis for $\cG$. The set of vectors $\{1_B\ |\ B \in \cB\}$ is an orthogonal set in $L^2(\cG, \mu_s)$. Furthermore, it is a basis in the following sense: for any $f \in L^2(\cG, \mu_s)$ we have the decomposition
    \begin{equation}
        f = \sum_{B \in \cB} f \cdot 1_B = \sum_{B \in \cB} (f \circ t|_B\inv) \ast 1_B = \sum_{B \in \cB} 1_B \ast (f \circ s|_B\inv)
    \end{equation}
    where $f \cdot 1_B$ denotes pointwise multiplication and convergence of the series is in $L^2$-norm and where $*$ denotes the convolution product: $f_1 * f_2 (g) = \sum_{h: s(h) = s(g)} f_1(h)f_2(h^{-1}g)$.
\end{proposition}
\begin{proof}
    The orthogonality of the $1_B$ follows from the following computation for bisections $B,C \in \cB$:
    \begin{equation*}
        \langle  1_B, 1_C \rangle = \langle \lambda_B 1_{\cG\zero}, \lambda_C 1_{\cG\zero}  \rangle = \langle \lambda_{C\inv} \lambda_B 1_{\cG\zero},  1_{\cG\zero}  \rangle = \mu\left(\left( C\inv \cdot B\right) \cap \cG\zero\right)
    \end{equation*}
    and $\left(C\inv \cdot B\right) \cap \cG\zero = \emptyset$ whenever $B \neq C$. The decomposition of $f \in L^2(\cG, \mu_s)$ follows from the fact that $(\cG, \mu_s) = \left(\bigsqcup_{B \in \cB} B, \sum_{B \in \cB} \mu_s |_B\right)$. Finally, a computation shows that for any $f \in \cB$ it holds that
    \begin{equation*}
        f\cdot 1_B  = (f \circ t|_B\inv) \ast 1_B = 1_B \ast (f \circ s|_B\inv).\qedhere
    \end{equation*}
\end{proof}

Before we get to the Fourier decomposition, we recall the definition of $j$ map and the sharp norm $\| \cdot \|^\sharp$. Let $(\cG, \mu)$ be a discrete measured groupoid and denote by
\begin{equation*}
    A(\cG, \mu) := \Span \left\{\lambda_B\ |\ B \subset \cG \text{ is a Borel bisection } \right\},
\end{equation*}
which is called the \emph{Steinberg algebra}. Recall that $A(\cG, \mu)$ is SOT/WOT-dense in $L(\cG)$. One can embed $L(\cG)$ into $L^2(\cG, \mu_s)$ via the map
\begin{equation}\label{eq:j map}
    j: L(\cG) \to L^2(\cG, \mu_s) : a \mapsto j(a) := a 1_{\cG\zero}.
\end{equation}
It is not difficult to check that $j$ is a continuous linear map. To see that $j$ is injective, denote by $E$ the conditional expectation from $L(\cG)$ onto $L^\infty(\cG\zero, \mu)$ and define the state $\phi_\mu$ on $A(\cG,\mu)$ by setting
\begin{equation*}
    \phi_\mu(\lambda_A) = \int E(1_A) \mathrm{d}\mu = \mu(A\cap \cG\zero)
\end{equation*}
whenever $A$ is a Borel subset of $\cG$ with $\mu_s(A) < \infty$. Then $\phi_\mu$ extends to a faithful normal state on $L(G)$ by $\phi_\mu(a) = \int E(a) \mathrm{d}\mu$. Now take $a$ and $b$ in $L(\cG)$ such that $j(a) = j(b)$. Then
\begin{align*}
    \phi_\mu((a-b)^\ast(a-b)) &= \int E((a-b)^\ast(a-b)) \mathrm{d}\mu\\
    &= \int E((a-b)^\ast(a-b))1_{\cG\zero} \mathrm{d}\mu\\
    &= \int E((a-b)^\ast(a-b)1_{\cG\zero}) \mathrm{d}\mu\\
    &= \int E((a-b)^\ast(j(a) - j(b)) \mathrm{d}\mu = 0 \;.
\end{align*}
Hence the faithfulness of $\phi_\mu$ implies that $a-b = 0$ so that $j$ is injective. 

On $A(\cG,\mu)$ one sees that $j$ acts as the identity, by which we mean that $j(\lambda_A) = 1_A$ for all Borel subsets $A \subset \cG$ with $\mu_s(A) < \infty$. In fact, it is not difficult to check that $j(\lambda_A x \lambda_B) = \lambda_A \rho_{B\inv} j(x)$ for all Borel bisections $A, B \subset \cG$ and all $x \in L(\cG)$.

Note that the GNS representation associated to $\phi_\mu$ is exactly the representation of $L(\cG)$ on $L^2(\cG, \mu_s)$ given by \eqref{equation:Left regular rep}. In particular, it holds that
\begin{equation*}
    \phi_\mu(a^*a) = \langle j(a), j(a) \rangle_{L^2}
\end{equation*}
for all $a \in L(\cG)$. We define the \emph{sharp norm} of an element $a \in L(\cG)$ as
\begin{equation}
    \| a \|^\sharp = \sqrt{\phi_\mu(a^*a) + \phi_\mu(aa^*)}.
\end{equation}
By \cite[Chapter III, Proposition 5.3]{Takesaki01} the sharp norm agrees with the $\sigma$-strong $\ast$-topology on bounded sets.

\begin{lemma}\label{Lemma:a^B and a_B}
    Let $(\cG, \mu)$ be a discrete measured groupoid and let $a \in L(\cG)$. For each Borel bisection $B \subset \cG$, the functions
    \begin{equation*}
        a_B := (a 1_{\cG\zero}) \circ s|_B\inv \qquad \text{ and } \qquad a^B := (a 1_{\cG\zero}) \circ t|_B\inv
    \end{equation*}
    belong to $L^\infty(\cG\zero, \mu)$. In particular, $j(a)|_B \in L^\infty(B, \mu_s)$.
\end{lemma}
\begin{proof}
    First we note that if $E \subset \cG$ is a bisection with $\mu_s(E) < \infty$, then
    \begin{equation*}
        (\lambda_E \cdot 1_{\cG\zero}) \circ s|_B\inv = 1_E \circ s|_B\inv = 1_{s(E \cap B)} \in L^\infty(\cG\zero, \mu).
    \end{equation*}
    It follows that for every $a \in A(\cG, \mu)$ we have that $a_B \in L^\infty(\cG\zero, \mu)$.

    Now suppose $a \in L(\cG)$ with $\|a \| \leq 1$. Let $(a_n)_n$ be a sequence in $A(\cG, \mu)$ be a sequence with $\|a_n\| \leq 1$ for all $n$ and such that $a_n \to a$ in the SOT. By definition, we get that $a_n 1_{\cG\zero} \to a 1_{\cG\zero}$ in $L^2$-norm. Then also $a_n 1_{\cG\zero} \circ s|_B\inv \to a 1_{\cG\zero} \circ s|_B\inv$ in $L^2$-norm. Now it suffices to note that from $\|a\| \leq 1$ it follows that $\|a_{n,B} \|_\infty = \|a_n 1_{\cG\zero} \circ s|_B\inv\|_\infty \leq 1$. Then, since the $\|\cdot \|_2$ and the $\sigma$-strong* topology agree on bounded sets, we get that $a_{n,B} \to a_B$ in the SOT, which implies that $a_B \in L^\infty(\cG\zero, \mu)$ since all $a_{n,B}\in L^\infty(\cG\zero, \mu)$ by the first part of the proof.

    Finally, to see that also $a^B \in L^\infty(\cG\zero, \mu)$, simply note that $a^B = a_B \circ \sigma_B\inv$.
    \end{proof}

\begin{proposition}\label{Proposition:Fourier decomposition}
    Let $(\cG, \mu)$ be a discrete measured groupoid and fix a symmetric basis $\cB$ of $\cG$. Denote by $E: L(\cG) \to L^\infty(\cG\zero, \mu)$ the conditional expectation. Any element of $a \in L(\cG)$ admits a Fourier decomposition
    \begin{equation*}
        a = \sum_{B \in \cB} E(a \lambda_B^*) \lambda_B
    \end{equation*}
    which converges under the sharp norm. Furthermore, the Fourier coefficients satisfy the identity
    \begin{equation*}
        E(a \lambda_B^*) = a^B
    \end{equation*}
    for all $B \in \cB$. In particular, it holds that $\sum_B \|E(a \lambda_B^*) \|_2^2 < \infty$.
\end{proposition}
\begin{proof}
    Let $a \in L(\cG)$ be arbitrary. Since $j(a) \in L^2(\cG, \mu_s)$, by Proposition \ref{Proposition:basis for L2} we know that
    \begin{equation*}
        j(a) = \sum_{B \in \cB} a^B \ast 1_B
    \end{equation*}
    with $a^B$ as in Lemma \ref{Lemma:a^B and a_B}. For any finite set $\cF \subset \cB$, set
    \begin{equation*}
        a_\cF := \sum_{B \in \cF} a^B\lambda_B.
    \end{equation*}
    We claim that the net $a_\cF$ converges to $a$ in the sharp norm. We first show that $\phi_\mu((a-a_\cF)^*(a- a_\cF)) \to 0$ as $\cF$ tends to infinity. This follows from noting that
    \begin{align*}
        \phi_\mu((a-a_\cF)^*(a- a_\cF)) = \langle j(a-a_\cF), j(a-a_\cF) \rangle &= \left\| \sum_{B \in \cB \setminus \cF} a^B \ast 1_B \right\|_2
    \end{align*}
    To see that $\phi_\mu((a-a_\cF)(a-a_\cF)^*)) \to 0$ as $\cF$ tends to infinity, note that 
    \begin{equation*}
        (a-a_\cF)(a-a_\cF)^* = (a^* - a_\cF^*)^*(a^* - a_\cF^*)
    \end{equation*}
    so it suffices to check that $(a_\cF)^*$ gives the appropriate Fourier decomposition of $a^*$. To this end, a simple calculation shows that
    \begin{equation}
        (a^*)^B = \overline{a^B} \circ \sigma_{B\inv}\inv
    \end{equation}
    for every $B \in \cB$. Then we see that
    \begin{equation}
        a_\cF^* = \sum_{\cB \in \cF} \lambda_{B\inv} \overline{a^B} = \overline{a^B} \circ \sigma_{B\inv}\inv \lambda_{B\inv}.
    \end{equation}
    Since the basis $\cB$ is symmetric, we see that $(a_\cF)^*$ and $(a^*)_\cF$ agree in the limit, which is what we needed to show.

    Finally, to see that the Fourier coefficients satisfy $E(a\lambda_B^*) = a^B$, fix a $B \in \cB$ and observe that whenever $\cF \subset \cB$ is a finite set containing $B$, it holds that $E(a_\cF \lambda_B^*) = a_\cF^B$. The $E( \cdot \lambda_B^*)$ is a normal map, hence $\| \cdot \|^\sharp$-continuous on bounded subsets. Since the net $(a_\cF)$ is bounded in sharp norm by $\|a\|^\sharp$, we indeed get that $(a\lambda_B^*) = \lim_\cF E(a_\cF \lambda_B^*) = a^B$.
\end{proof}

\begin{remark}
    Proposition \ref{Proposition:Fourier decomposition} remains valid if the basis $\cB$ is not symmetric, but in that case some of the computations are more involved.
\end{remark}

\begin{lemma}
    Let $\cG$ be a discrete Borel groupoid and let $\cB$ be a basis for $\cG$. Suppose $C \in \cB$ is a bisection. The set of conjugates $C\cB C\inv = \{CBC\inv\ |\ B \in \cB \}$ also forms a basis for $C \cG C^{-1}$. If $\cB$ is symmetric, then so is $C\cB C\inv$.
\end{lemma}
\begin{proof}
 First let us show that elements of $C\cB C^{-1}$ are mutually disjoint. Let $B,D \in \cB$. Suppose that $a,b,c,d \in C$ and $x \in B$ and $y \in D$ such that $axb^{-1} = cyd^{-1}$. Note in particular that $t(a) = t(c)$ and that $t(b) = t(d)$. However, each $a,b,c,d$ belong to the same bisection. By injectivity of the target map, we get that $a = c$ and $b=d$. Thus, $x = a\inv a x b\inv b = a\inv c y d\inv b = c\inv c y d\inv d = y \in B \cap D$. Therefore, $B = D$. Finally, observe that
    \begin{equation*}
        \bigcup_B CBC^{-1} = C \left(\bigcup_B B \right) C^{-1} = C \cG C^{-1}\;. \qedhere
    \end{equation*}
\end{proof}

\begin{lemma}\label{Lemma:Basis of Iso}
    Let $\cG$ be a discrete Borel groupoid. There is a countable set $\cD$ of mutually disjoint bisections of $\cG$ such that for any basis $\cB$ of $\Iso(\cG)$, the set $\cB \cup \cD$ is a basis for $\cG$.
\end{lemma}

\begin{proof}
    This is a straightforwared modification of the proof of Proposition \ref{Proposition:basis of groupoid}.
\end{proof}

\begin{lemma}\label{Lemma:Fourier decomp when support on Iso}
    Let $(\cG, \mu)$ be a discrete measured groupoid and let $\cB$ be a basis for $\Iso(\cG)$. For every $a \in L(\cG)$ such that $j(a)$ is supported on $\Iso(\cG)$, we have the Fourier decomposition
    \begin{equation*}
        a = \sum_{B \in \cB} E(a\lambda_B^*) \lambda_B.
    \end{equation*}
\end{lemma}
\begin{proof}
    This follows easily from Proposition \ref{Proposition:Fourier decomposition} together with Lemma \ref{Lemma:Basis of Iso}.
\end{proof}

\begin{lemma}\label{Lemma:support j(a)}
    Let $(\cG, \mu)$ be a discrete measured groupoid. If $a \in \cZ(L(\cG))$, then $j(a)$ is supported on $\Iso(\cG)$.
\end{lemma}
\begin{proof}
    Let $a$ be a central element of $L(\cG)$. Fix a basis $\cB$ of $\cG$ and decompose $j(a)$ as
    \begin{equation}
        j(a) = \sum_{B \in \cB} 1_B \ast a_B.
    \end{equation}
    We argue by contradiction. Suppose $j(a)$ is not supported on $\Iso(\cG)$. Then we find a non-null Borel subset $E \subset \cG \setminus \Iso(\cG)$ such that $j(a)(g) \neq 0$ for all $g \in E$. Since $\cB$ is countable and all $B \in \cB$ are mutually disjoint, there is a $B_0 \in \cB$ such that $\mu_s(E \cap B_0) >0$. Note that for every $x \in s(E \cap B_0)$ we have that $a_{B_0}(x) \neq 0$. By the centrality of $a$, we get for any Borel $F \subset s(E \cap B_0)$ that $j(\lambda_F a) = j(a \lambda_F)$. Since $F \subset \cG\zero$ and $a_B \in L^\infty(\cG\zero, \mu)$, a calculation shows 
    \begin{equation*}
         j(\lambda_F a) = \sum_{B \in \cB} 1_{F \cdot B} \ast a_B = \sum_B 1_B \ast (1_F \circ \sigma_B)  a_B.
    \end{equation*}
    On the other hand, $j(a \lambda_F) = \sum_{B \in \cB} 1_B \ast  a_B 1_F$. Since $a_B  1_F$ and $(1_F \circ \sigma_B)  a_B$ are both elements of $L^\infty(\cG\zero, \mu)$, looking at the restriction of $j(\lambda_F a)$ on each basis element gives us the identity $a_B  1_F = (1_F \circ \sigma_B)  a_B$ for all $B \in \cB$. Hence in particular $1_F =(1_F \circ \sigma_{B_0})$, meaning that for any $x \in s(E \cap B_0)$, $x \in F$ if and only if $\sigma_{B_0}(x) \in F$. Since this holds for any Borel subset $F \subset s(E \cap B_0)$, we get that $x = t \circ s|_{B_0}\inv (x)$ for all $x \in s(E \cap B_0)$. In other words, $E \cap B_0 \subset \Iso(\cG)$, which is a contradiction.
\end{proof}

\begin{lemma}\label{Lemma:j(a) conj invariant}
    Let $(\cG, \mu)$ be a discrete measured groupoid. If $a \in \cZ(L(\cG))$, then $j(a)(ghg\inv) = j(a)(h)$ for almost every $h \in \Iso(\cG)$ and all $g \in \cG$ with $s(g) = s(h)$.
\end{lemma}
\begin{proof}
Fix a central element $a \in L(\cG)$ and elements $g \in \cG$ and $h \in \Iso(\cG)$. Take bisections $B \subset \cG$ and $C \subset \Iso(\cG)$ containing $g$ and $h$ respectively. Fix a basis $\cB$ of $\Iso(\cG)$ containing $C$.

One the one hand, we have
\begin{equation*}
    j(\lambda_B a \lambda_{B\inv}) = \sum_{D \in \cB} (a^D \circ \sigma_B\inv) \ast 1_{BDB\inv}.
\end{equation*}
One the other hand, note that $B\cB B\inv$ is a basis of $\Iso(\cG)_{t(B)}$, which is the subgroupoid of $\Iso(\cG)$ with unit space $t(B)$. Thus we have the decomposition
\begin{equation*}
    j(\lambda_{t(B)} a) = \sum_{D \in \cB} a^{BDB\inv} \ast 1_{BDB\inv}.
\end{equation*}
Since $a$ is central, $\lambda_B a \lambda_{B\inv} = \lambda_B \lambda_{B\inv} a = \lambda_{t(B)} a$. Looking at the base coordinate $C$, this in particular means that
\begin{equation*}
    a^{BCB\inv} = a^C \circ \sigma_B\inv = a^C \circ s \circ t|_B\inv.
\end{equation*}
Now note that 
\begin{equation}
    (a^C \circ s \circ t|_B\inv)(t(ghg\inv)) = a^C(s(g)) = a^C(t(h)),
\end{equation}
hence we get that $a^{BCB\inv} (t(ghg\inv)) = a^C(t(h))$. Now it suffices to note that whenever an element $k \in \cG$ lies in some bisection $E$, one has that $j(a)(k) = a^E \ast 1_E(k) = a^E(t(k))$.
\end{proof}

\section{Factoriality of the untwisted groupoid von Neumann algebra}\label{Section:Factoriality untwisted}

In this section we prove Theorem \ref{theorem:Factoriality untwisted}. First we note that for any Borel subset of $\Iso(\cG)$, its `conjugacy class' is a Borel set.

\begin{lemma}\label{Lemma:conjugacy class Borel}
    Let $(\cG, \mu)$ be a discrete measured groupoid. For every Borel subset $A \subset \Iso(\cG)$ the set
    \begin{equation}\label{equation:conjugacy class}
        \Omega_A = \bigcup_{g \in \cG} gAg\inv
    \end{equation}
    is Borel.
\end{lemma}

\begin{proof}
    Let $A$ be any Borel subset of $\Iso(\cG)$ and fix a basis $\cB$ of $\cG$. We claim that
    \begin{equation*}
        \bigcup_{g\in  \cG} gAg\inv = \bigcup_{B \in \cB} BAB\inv,
    \end{equation*}
    from which the lemma follows, since this is a countable union of Borel sets. It is clear that $\Omega_A\subset \bigcup_{B \in \cB} BAB\inv$. We prove the other inclusion. 
    
    Assume that $x\in BAB\inv$ for some $B\in \mathcal{B}$. Then there exist $g,h\in B$ and $y\in A$ such that $x=gyh\inv$. Since $A\subset \Iso(\cG)$ we get that $s(g)=s(h)$ and because $B$ is a bisection we get that $g=h$. This proves that $x\in \Omega_A$. 
\end{proof}

In light of the previous lemma, we can introduce a notion of being icc for groupoids. Recall that a discrete group $\Gamma$ is icc if every element apart from the identity has an infinite conjugacy class. The following definition is a generalization of this property for discrete measured groupoids.

\begin{definition}\label{Definition:Finite measure conjugacy class and icc}
    Let $(\cG, \mu)$ be a discrete measured groupoid. For a Borel subset $A \subset \Iso(\cG)$, we call the set $\Omega_A$ from \eqref{equation:conjugacy class} its conjugacy class. We say $A$ has a \emph{$\mu$-finite conjugacy class} if $\mu_s(\Omega_A) < \infty$.

    We say $(\cG, \mu)$ has \emph{$\mu$-infinite conjugacy classes}, or simply that it is \emph{icc}, if whenever $A \subset \Iso(\cG)$ is a nontrivial bisection (or equivalently, whenever $A$ is a non-null Borel subset) with finite measure conjugacy class, then $A \subset \cG\zero$. 
\end{definition}

\begin{remark}
    Since the source and target map coincide on $\Iso(\cG)$ we get that $\mu_s(\Omega_A)=\mu_t(\Omega_A)$ for every bisection $A\subset \Iso(\cG)$. In partcicular, it does not matter if we use the measure $\mu_s$ or $\mu_t$ in Definition \ref{Definition:Finite measure conjugacy class and icc} above. 
\end{remark}


\begin{lemma}\label{lem:ergodic union of bisections}
    Let $(\cG, \mu)$ be an ergodic measured groupoid and let $A \subset \Iso(\cG)$ be a Borel subset with $\mu$-finite conjugacy class. Then there exist finitely many mutually disjoint bisections $V_1,...,V_k$ of measure one such that $\Omega_A = \bigsqcup_i V_i$ up to a null set. In particular, $\mu_s(\Omega_A) = k$.
\end{lemma}

\begin{proof}
    Let $A \subset \Iso(\cG)$ be as in the statement of the lemma. Define for each $k \in \N$ the set
    \begin{equation*}
        E_k := \left\{x \in \cG\zero\ \left| |s\inv(x) \cap \Omega_A| = k \right. \right\}.
    \end{equation*}
    Let $q\colon \cR_\cG\rightarrow \cG$ be a Borel section for the projection map $g\mapsto (t(g),s(g))$. Take an arbitrary $\vphi\in [\cR_\cG]$ and define $W=\{q((\vphi(x),x))\mid x\in \cG^{(0)}\}$. Then $W\in [\cG]$ and we have that $t(g)=\vphi(s(g))$ for every $g\in W$. Observe that
    \begin{align}
        W(s^{-1}(x)\cap \Omega_A)W^{-1}=Ws^{-1}(x)W^{-1}\cap \Omega_A = s^{-1}(\vphi(x))\cap \Omega_A\,,
    \end{align}
    so that $\vphi^{-1}(E_k)=E_k$. Thus we have shown that the sets $E_k$ are $\cR_\cG$-invariant. 
    
  By ergodicty we get that for each $k$ either $\mu(E_k) = 0$ or $\mu(E_k) = 1$. Since $\cG\zero = \bigsqcup_k E_k$, we find a $k$ for which $\mu(E_k) = 1$. 
  
  By \cite[Exercise 18.15]{Kec95}, there exists a bisection $V_1 \subseteq \Omega_A \cap s^{-1}(E_k)$ such that $s(V_1) = E_k$. Note that by definition of $E_k$, we have $s((\Omega_A \cap s^{-1}(E_k)) \setminus V_1) = E_k$. If $k=1$, then we are done. Otherwise, we work by induction: assume for some $j \leq k$ that we have mutually disjoint bisections $V_1,\ldots, V_j \subseteq \Omega_A \cap s^{-1}(E_k)$ with $s(V_i) = E_k$ for all $i$. If $j = k$, then we would have the desired bisections. Otherwise, note that the map $s: (\Omega_A \cap s^{-1}(E_k)) \setminus \left(\bigsqcup_{i \leq j} V_i\right) \to E_k$ must be a surjection by definition of $E_k$. Applying \cite[Exercise 18.15]{Kec95} again gives us a bisection $V_{j+1} \subseteq (\Omega_A \cap s^{-1}(E_k)) \setminus \left(\bigsqcup_{i \leq j} V_i\right)$ with $s(V_{j+1}) = E_k$. 
\end{proof}

\begin{remark}
    When the ergodicity assumption in Lemma \ref{lem:ergodic union of bisections} is removed, it is straightforward to come up with a counterexample. The transposition $(1,2)$ has conjugacy class of size ${n\choose 2}$ in the symmetric group $\cS_n$. 
    Then writing  $\rho=\sum_{n=2}^{\infty} 2^{-n}{n \choose 2}\inv<+\infty$ we get that $\mu(n)=\rho\inv2^{-n}{n \choose 2}\inv$ defines a probability measure $\mu$ on $\mathbb{N}_{\geq 2}$.
    Define $\cG=\{(n,\sigma)\mid n\geq 2 \text{ and } \sigma\in \cS_n\}$ with unit space $(\mathbb{N}_{\geq 2},\mu)$ and source and target map given by $(n,\sigma)\mapsto n$. The conjugacy class $\Omega_A$ of the bisection $A=\bigcup_{n\geq 2}\{(n,(1,2))\}$ satisfies $\mu(\Omega_A)<+\infty$ by construction, but it cannot be written as the union of finitely many bisections. 
\end{remark}

It immediately follows from the proof of Lemma \ref{lem:ergodic union of bisections} above that an ergodic discrete measured groupoid $\cG$ is icc if and only if for every bisection $A\subset \Iso(\cG)$ we have that $|s^{-1}(x)\cap\Omega_A|=+\infty$ for a.e.\ $x\in \cG\zero$, thus depending only on the \emph{measure class} of $\mu$. In general we have the following characterization that is independent of the choice of equivalent measure.

\begin{proposition}\label{prop:measure class icc}
Suppose that $(\cG,\mu)$ is a discrete measured groupoid. For a bisection $A\subset \Iso(\cG)$ and an element $x\in \cG\zero$ write $A_x=\bigcup_{g\in t^{-1}(x)}gAg\inv=s^{-1}(x)\cap \Omega_A$. The following are equivalent:
\begin{enumerate}
    \item The groupoid $(\cG,\mu)$ is icc.
    \item For every bisection $A\subset \Iso(\cG)\setminus \cG\zero$ we have that $|A_x|=+\infty$ for a.e.\ $x \in s(A)$.
\end{enumerate}
\end{proposition}

\begin{proof}
    First assume that every bisection $A\subset \Iso(\cG)\setminus\cG\zero$ satisfies $|A_x|=+\infty$ for a.e.\ $x\in s(A)$. Then for every bisection $A\subset \Iso(\cG)\setminus\cG\zero$ with $\mu_s(A)>0$ we get that
    \begin{align}
        \mu_s(\Omega_A)=\int_{s(\Omega_A)}|A_x|\mathrm{d}\mu(x)\geq \int_{s(A)}|A_x|\mathrm{d}\mu(x)=+\infty\,,
    \end{align}
    so that $\cG$ is icc.

    Conversely, assume that $A\subset \Iso(\cG)\setminus \cG\zero$ is a bisection such that 
    $$\mu(\{x\in s(A)\mid|A_x|<+\infty\})>0\,.$$
    Then there exists $N\in \mathbb{N}$ large enough such that the set $B=\{g \in A\mid |A_{s(g)}|\leq N\}$ satisfies $\mu_s(B)>0$. Because $B\subset A$ we get that $|B_x|\leq |A_x|$ for every $x\in s(B)$. Furthermore, since $B_{k\cdot x}=kB_{x}k\inv$ we get that $|B_x|\leq N$ for every $x\in s(\Omega_B)$. Then we see that
    \begin{align*}
        \mu_{s}(\Omega_{B})=\int_{s(\Omega_B)}|B_x|\mathrm{d}\mu(x)\leq N\,,
    \end{align*}
    showing that $\cG$ is not icc. 
\end{proof}

Now we are ready to proof Theorem \ref{theorem:Factoriality untwisted}.

\begin{proof}[Proof of Theorem \ref{theorem:Factoriality untwisted}]
    First assume that condition 1 fails, so there is a Borel subset $A \subset \cG$ with $\mu_s(\Omega_A) < \infty$ and such that $A$ is not contained in $\cG\zero$. Since $\Omega_A \cdot D = D \cdot \Omega_A$ for all bisections $D \subset \cG$, it follows that $\lambda_{\Omega_A} \in \cZ(L(\cG))$. Since $A \cap \cG\zero \neq A$, it follows that $\lambda_{\Omega_A} \in L(\cG) \setminus L^\infty(\cG\zero, \mu)$.

    Now assume condition 2 fails. Then we find an element $a \in \cZ(L(\cG)) \setminus L^\infty(\cG\zero, \mu)$. Note that $j(a) \neq 0$ and we can assume that $j(a)|_{\cG\zero} = 0$. Write
    \begin{equation*}
        \supp(j(a)) =  \bigcup_{n \geq 1}  \left\{  g \in \cG\  \left|\ |j(a)(g)|^2 \geq \frac{1}{n^2} \right. \right\},
    \end{equation*}
    and fix an $n \in \N$ such that $E:= \left\{ g \in \cG\ |\ |j(a)(g)|^2 \geq \frac{1}{n^2} \right\}$ satisfies $\mu_s(E) > 0$. Note that $\mu_s(E \cap \cG\zero) = 0$. By Lemma \ref{Lemma:support j(a)} we know that $E \subset \Iso(\cG)$ and from Lemma \ref{Lemma:j(a) conj invariant} it follows that $g \cdot E \cdot g\inv \subset E$ for all $g \in \cG$. Finally, to see that $\mu_s(E) < \infty$, note that
    \begin{equation*}
        \mu_s(E) \leq n^2 \int_{\cG} | j(a)|^2 \mathrm{d}\mu_s = n^2 \| j(a)\|_2^2 < \infty.\qedhere
    \end{equation*}
\end{proof}

\begin{proof}[Proof of Corollary \ref{Corollary:Factoriality untwisted}]
    By Theorem \ref{theorem:Factoriality untwisted} it suffices to note that for any discrete measured groupoid $(\cG, \mu)$ it holds that $L^\infty(\cG\zero, \mu)^\cG = \C \cdot 1_{\cG\zero}$ if and only if $\cG$ is ergodic.  \qedhere
\end{proof}

\section{Factoriality of the twisted groupoid von Neumann algebra}\label{Section:Factoriality twisted}

\subsection{Definitions and basic properties}

Let $\cG$ be a discrete borel groupoid. A \emph{$2$-cocycle} on $\cG$ is a Borel map 
\begin{align*}
    \omega: \cG\two \to \T
\end{align*}
such that $\omega(x,yz)\omega(y,z) = \omega(xy,z)\omega(x,y)$ for all $(x,y,z) \in \cG\three$. If moreover $\omega(x,x^{-1}x) = 1 = \omega(xx^{-1},x)$ for all $x \in G$, we say that $\omega$ is \emph{normalized}. We call the pair $(\cG,\omega)$ a \emph{twisted discrete Borel groupoid}.

We say two cocycles $\omega_1$ and $\omega_2$ are \emph{cohomologous} if there is a Borel map $\rho: \cG \to \T$ such that $\omega_1(x,y) = \rho(x)\rho(y)\overline{\rho(xy)} \omega_2(x,y)$ for almost every $(x,y) \in \cG\two$. Every cocycle is cohomologous to a normalized cocycle which moreover satisfies $\omega(x, x\inv) = 1$ for all $x \in \cG$. To see this, first note that for any cocycle $\omega$ the cocycle identity implies that $\omega(x, x\inv x) = \omega(x\inv x, x\inv x)$ and $\omega(x x\inv, x) = \omega(x x\inv, x x\inv)$ for all $x \in \cG$. The cocycle 
\begin{equation*}
    \omega'(x,y) = \overline{\omega(x, x\inv x)}\ \overline{\omega(y, y\inv y)} \omega(xy, (xy)\inv xy) \omega(x,y)
\end{equation*} 
satisfies $\omega'(x\inv x, x\inv x) = 1 = \omega'(xx\inv, xx\inv)$ for all $x \in \cG$ and hence is normalized. Next, suppose $\omega$ is a normalized cocycle and consider
\begin{equation*}
\omega'(x,y)  = (\overline{\omega(x,x\inv)})^{1/2} (\overline{\omega(y, y\inv)})^{1/2} \omega(xy, (xy)\inv)^{1/2} \omega(x,y).
\end{equation*}
Note that $\omega(x, x\inv) = 1$ also implies that $\overline{\omega(x,y)} = \omega(y\inv, x\inv)$. Throughout this sections all cocycles are assumed to satisfy this property.

Let $(\cG, \mu, \omega)$ be a twisted measured discrete groupoid. To define the twisted groupoid von Neumann algebra, we first note the following: whenever $A, B \subset \cG$ are bisections, there is a function $\omega(A, B) \in L^\infty(\cG\zero, \mu)$ defined by
\begin{equation*}
    \omega(A,B)(x) = \begin{cases}
        \omega(a,b) & \text{where } a \in A, b \in B, s(a) = t(b), \text{ and } x = t(a)\\
        0 & \text{if } x \notin t(A \cdot B).
    \end{cases}
\end{equation*}
Now we proceed analogously as in the untwisted case by defining a `projective left regular representation' in the following way: for any Borel bisection $A \subset \cG$ we define the operator $\lambda_A^\omega \in B(L^2(\cG, \mu_s))$ by
\begin{equation*}
    \lambda_A^\omega 1_E(g) = \omega(g_1, g_2) 1_{A \cdot E}(g) \qquad \text{where } g_1 \in A, g_2 \in E \text{ such that } g_1g_2 = g
\end{equation*}
for any subset $E \subset \cG$ with $\mu_s(E) < \infty$. The \emph{(left) twisted groupoid von Neumann algebra} is defined as  $L_\omega(\cG) := \{\lambda_A^\omega\ |\ A \subset \cG \text{ Borel bisection} \}''$. Note that, as in the untwisted case, $L_\omega(\cG)$ includes a copy of $L^\infty(\cG\zero, \mu)$ as a regular subalgebra. For any Borel bisection $A$ one has $(\lambda_A^\omega)^* = \lambda_{A\inv}^\omega$ and one can check that $\lambda_A^\omega \lambda_B^\omega = \omega(A,B)\lambda_{A \cdot B}^\omega$ for any two Borel bisections $A, B \subset \cG$.

Similarly, one can define a projective right regular representation of $\cG$ by defining for any Borel bisection $A \subset \cG$ the operator
\begin{equation}
    \rho_A^\omega 1_E(g) = \overline{\omega(g_2, g_1)}1_{E \cdot A\inv}(g) \qquad \text{where } g_1 \in A\inv, g_2 \in E \text{ such that } g_2g_1 = g.
\end{equation}
One denotes the \emph{right twisted groupoid von Neumann algebra} by $R_\omega(\cG) := \{\rho_A^\omega\ |\ A \subset \cG \text{ is a Borel bisection} \}''$. We again have the relations $(\rho_A^\omega)^* = \rho_{A\inv}^\omega$ and $\rho_A^\omega \rho_B^\omega = \omega(A,B) \rho_{A \cdot B}^\omega$. We now show that the representations $\lambda^\omega$ and $\rho^{\overline{\omega}}$ are commutants of each other. For $f, g \in L^2(\cG, \mu_s)$ we define the twisted convolution product $\ast_\omega$ by
\begin{equation*}
    (f \ast_\omega g)(x) = \sum_{a,b; ab = x} \omega(a, b) f(a) g(b).
\end{equation*}
An application of H\"older's inequality shows that $\|f \ast_\omega g\|_\infty \leq \|f\|_2 \|g\|_2$. In particular, $f \ast_\omega g \in L^\infty(\cG, \mu_s)$.

We say a function $f \in L^2(\cG)$ is a \emph{left convolver} for $(\cG, \omega)$ if $f \ast_\omega g \in L^2(\cG, \mu_s)$ for every $g \in L^2(\cG, \mu_s)$. In that case it follows that $L_f^\omega(g) = f \ast_\omega g$ defines a bounded operator on $L^2(\cG, \mu_s)$. We denote by $LC_\omega(\cG)$ the set consisting of all $L_f^\omega$ for left convolvers $f$. Similarly, one defines the right convolvers $RC_\omega(\cG)$.

\begin{proposition}\label{Proposition:Commutant of left regular representation}
    Let $(\cG, \omega)$ be a twisted discrete measured groupoid. Then $L_\omega(\cG) = \LC_\omega(\cG) = R_{\overline{\omega}}(\cG)'$ and  $R_{\overline{\omega}}(\cG) = \RC_{\overline{\omega}}(\cG) = L_\omega(\cG)'$.
\end{proposition}
\begin{proof}
    We show the sequence of inclusions $R_{\overline{\omega}}(\cG)' \subseteq LC_\omega(\cG) \subseteq L_\omega(\cG) \subseteq R_{\overline{\omega}}(\cG)'$ which shows the first string of identities. The second part is treated analogously.

    For the first inclusion, let $T \in R_{\overline{\omega}}(\cG)'$ and define $f:= T(1_{\cG\zero})$. We claim that $T = L_f^\omega$. It suffices to show equality on $1_B$ for $B$ an arbitrary bisection. First, note that
    \begin{equation*}
        T(1_B) = T\rho_{B\inv}^{\overline{\omega}}(1_{\cG\zero}) = \rho_{B\inv}^{\overline{\omega}}T(1_{\cG\zero}) = \rho_{B\inv}^{\overline{\omega}}(f).
    \end{equation*}
    Next, observe that for any functions $h \in L^2(\cG, \mu_s)$ one has that $\rho_{B\inv}^{\overline{\omega}}(h) = h\ast_\omega 1_B$. Thus we find that
    \begin{equation}
        \rho_{B\inv}^{\overline{\omega}}(f) = f\ast_\omega 1_B = L_f^\omega(1_B),
    \end{equation}
    which shows the claim and thus the first inclusion.

    For the second inclusion, let $f \in L^2(\cG)$ be a left convolver and fix a basis $\cB$ of $\cG$. Let $\scF$ be a directed set of finite subsets of $\cB$. For each $\cF \in \scF$, define $f_\cF := \sum_{B \in \cF} f_B \ast_\omega 1_B$ where $f_B := f \circ t|_B\inv$. By construction, the operator $\lambda_{f_\cF}^\omega$ defined by
    \begin{equation*}
        \lambda_{f_\cF}^\omega \xi (g) = \sum_{x,y; xy = g} \omega(x,x\inv y) f(x) \xi(x\inv y)
    \end{equation*}
    belongs to $L_\omega(\cG)$. Note that 
    \begin{align*}
        \|\lambda_{f_\cF}^\omega\|^\sharp = \sqrt{\|f_\cF\|_2^2 + \|f_\cF^*\|_2^2} \leq \sqrt{\|f\|_2^2 + \|f^*\|_2^2}\;. 
    \end{align*}
    As well, for any $\cE,\cF \in \scF$, 
    \begin{align*}
        \|\lambda_{f_\cE}^\omega - \lambda_{f_\cF}^\omega\|^\sharp = \sqrt{\|f_\cE - f_\cF\|_2^2 + \|f^*_\cE - f^*_\cF\|_2^2} \;. 
    \end{align*}
    In particular as $f_\cF \to f$ and $f_{cF}^* \to f^*$ in $L^2$-norm, the net $(\lambda_{f_\cF}^\omega)_{\mathcal{F} \in \scF}$ is Cauchy. Since these elements belong in a fixed closed ball, $\lambda_{f_\cF}^\omega$ converges in the $\sigma$-strong* operator topology to some $T \in L_\omega(\cG)$. We claim that $T = L_f^\omega$.  For this, it suffices to show that these operators agree on bisections. For a bisection $B$,  since the $\sigma$-strong* operator topology is stronger than the strong operator topology, we have the limit $T(1_B) = \lim_\cF \lambda_{f_\cF}^\omega(1_B)$, and 
    \begin{align*}
        \|L_f^\omega(1_B) - \lambda_{f_\cF}^\omega(1_B)\|_2 = \|(f - f_\cF) \ast_\omega 1_B\|_2 \leq \|\rho_{B\inv}^\omega\|\|f - f_\cF\|_2 \;. 
    \end{align*}
    Thus, $\lim_{\cF} \lambda_{f_\cF}^\omega(1_B) = L_f^\omega(1_B)$, and we have the desired identity.

    For the last inclusion, it suffices to note that for any bisections $A, B$ one has that $\lambda_A^\omega \rho_B^{\overline{\omega}}(f) = 1_A \ast_\omega f \ast_\omega 1_{B\inv} = \rho_B^{\overline{\omega}} \lambda_A^\omega (f)$, so that $L_\omega(\cG) \subset R_{\overline{\omega}}(\cG)'$.
\end{proof}

\subsection{Fourier decomposition and the \texorpdfstring{$j$}{j} map again}

The results in Section \ref{Section:Basis, Fourier, j map} still remain valid in the twisted case, albeit slightly modified to include the twist when necessary. Let us make this precise. First, we note that the map
\begin{equation*}
    j: L_\omega(\cG) \to L^2(\cG, \mu_s): j(a) = a1_{\cG\zero}
\end{equation*}
still defines an injective linear map sending $\lambda_A^\omega$ to $1_A$ whenever $A \subset \cG$ is a Borel bisection. As before, one can check that $j(\lambda_A^\omega x \lambda_B^\omega) = \lambda_A^\omega \rho_{B\inv}^{\overline{\omega}} j(x)$. We again consider the faithful normal state $\phi_\mu$ on $ L_{\omega}(\cG)$ uniquely determined by
\begin{equation*}
    \phi_\mu(\lambda_A) = \int E(1_A) \mathrm{d}\mu = \mu(A\cap \cG\zero) \text{ for every bisection } A\subset \cG\,, 
\end{equation*}
and we set
\begin{equation*}
    \| a \|^\sharp = \sqrt{\phi_\mu(a^*a) + \phi_\mu(aa^*)}
\end{equation*}
for $a \in L_\omega(\cG)$. Lemma \ref{Lemma:a^B and a_B} remains valid, as does the Fourier decomposition of Proposition \ref{Proposition:Fourier decomposition}. We give the precise statement in the proposition below but omit the proof since it carries over from Proposition \ref{Proposition:Fourier decomposition} verbatim.

\begin{proposition}\label{Proposition:Fourier decomposition twisted}
    Let $(\cG, \mu, \omega)$ be a twisted discrete measured groupoid and fix a symmetric basis $\cB$ of $\cG$. Denote by $E: L_\omega(\cG) \to L^\infty(\cG\zero, \mu)$ the conditional expectation. Any element of $a \in L_\omega(\cG)$ admits a Fourier decomposition
    \begin{equation*}
        a = \sum_{B \in \cB} E(a (\lambda_B^\omega)^*) \lambda_B^\omega
    \end{equation*}
    which converges in the sharp norm. Furthermore, the Fourier coefficients satisfy the identity
    \begin{equation*}
        E(a (\lambda_B^\omega)^*) = a^B
    \end{equation*}
    for all $B \in \cB$. In particular, it holds that $\sum_B \|E(a (\lambda_B^\omega)^*) \|_2^2 < \infty$.
\end{proposition}

Lemmas \ref{Lemma:Fourier decomp when support on Iso}, \ref{Lemma:support j(a)}, and \ref{Lemma:j(a) conj invariant} translate to the following results in the twisted case.

\begin{lemma}
    Let $(\cG, \mu, \omega)$ be a twisted discrete measured groupoid and let $\cB$ be a basis for $\Iso(\cG)$. For every $a \in L_\omega(\cG)$ such that $j(a)$ is supported on $\Iso(\cG)$, we have the Fourier decomposition
    \begin{equation*}
        a = \sum_{B \in \cB} E(a(\lambda_B^\omega)^*) \lambda_B^\omega.
    \end{equation*}
\end{lemma}
\begin{proof}
    This is an application of Proposition \ref{Proposition:Fourier decomposition twisted} together with Lemma \ref{Lemma:Basis of Iso}.
\end{proof}

Lastly, the analogue of Lemma \ref{Lemma:j(a) conj invariant} in the twisted setting becomes the following lemma. The proof is very similar to the proof of Lemma \ref{Lemma:j(a) conj invariant}, but for the convenience of the reader we give all details.

\begin{lemma}\label{lemma:j(a) con inv twisted}
    Let $a \in L_\omega(\cG)$. The following are equivalent: 
    \begin{enumerate}
        \item $a$ belongs to the center of $L_\omega(\cG)$. 
        \item $j(a)$ is supported on $\Iso(\cG)$ and the identity 
    \begin{equation}\label{eq:j(a) identity twisted}
        \omega(ghg\inv, g) j(a)(ghg\inv) = \omega(g, h) j(a)(h)
    \end{equation}
    holds for almost every $h \in \Iso(\cG)$ and every $g \in \cG$ with $s(g) = s(h)$.
    \end{enumerate} 
\end{lemma}
\begin{proof}
    First assume that $a \in \cZ(L_\omega(\cG))$. That $j(a)$ is supported on $\Iso(\cG)$ follows exactly as in the untwisted case. Take $g \in \cG$ and $h \in \Iso(\cG)$ with $s(g) = s(h)$ and let $B \subset \cG$ and $C \subset \Iso(\cG)$ be bisections containing $g$ and $h$ respectively. If $h \in \cG\zero$, take $C = \cG\zero$ and if $h \notin \cG\zero$, assume that $C \cap \cG\zero = \emptyset$. Fix a basis $\cB$ of $\Iso(\cG)$ that contains $C$. On the one hand, we have
    \begin{equation*}
        j\left(\lambda_B^\omega a \lambda_{B\inv}^\omega \right) = \rho_B^{\overline{\omega}} \lambda_B^\omega j(a) = \sum_{D \in \cB} \omega(B, D)\omega(BD, B\inv) (a^D \circ \sigma_B\inv) \ast 1_{BDB\inv}.
    \end{equation*}
    On the other hand, using the fact that $B\cB B\inv$ is a basis for $\Iso(\cG)_{t(B)}$, we have the decomposition
    \begin{equation*}
        j\left(\lambda_{t(B)}^\omega a\right) = \sum_{D\in \cB} a^{BDB\inv} \ast 1_{BDB\inv}.
    \end{equation*}
    Looking at the base coordinate $C$, we get the identity
    \begin{equation}\label{eq:bisection omega regularity}
        a^{BCB\inv} = \omega(B,C)\omega(BC, B\inv) a^C \circ \sigma_B\inv.
    \end{equation}
    Note that
    \begin{equation*}
        a^C \circ \sigma_B\inv (t(ghg\inv)) = a^C(t(h)),
    \end{equation*}
    as well as
    \begin{equation*}
        \omega(B,C)(t(ghg\inv)) = \omega(B,C)(t(g)) = \omega(g, h)
    \end{equation*}
    and
    \begin{equation*}
        \omega(BC, B\inv)(t(ghg\inv)) = \omega(gh, g\inv).
    \end{equation*}
    Hence we get that
    \begin{equation}
        a^{BCB\inv}(t(ghg\inv)) = \omega(g,h)\omega(gh, g\inv) a^C (t(h)).
    \end{equation}
    Using the cocycle identity it follows that $\omega(gh, g\inv) = \overline{\omega(ghg\inv, g)}$ so that
    \begin{equation*}
        \omega(ghg\inv, g) a^{BCB\inv}(t(ghg\inv)) = \omega(g,h) a^C(t(h)).
    \end{equation*}
    Finally, note that whenever $k$ lies in some bisection $E$, one has that $j(a)(k) = a^E \ast 1_E(k) = a^E(t(k))$, from which \eqref{eq:j(a) identity twisted} follows.

    Conversely, let $B \subseteq \cG$ be any bisection. We first show that $j(\lambda_B^\omega a \lambda_{B^{-1}}^\omega) = j(\lambda_{t(B)}^\omega a)$. For any $h \in \Iso(\cG)$, letting $g \in B$ be such that $s(g) = t(h)$ if $g$ exists, 
    \begin{align*}
        j(\lambda_B^\omega  a \lambda_{B^{-1}}^\omega)(h) &= 1_{t(B)}(h)\omega(g^{-1},gh)\omega(g^{-1},h)j(a)(ghg^{-1})\\
        &= \omega(g^{-1},g)\omega(g^{-1}g,h)j(a)(h) e_{t(B)}(h) \\
        &= j(\lambda_{t(B)}^\omega a)(h)\;. 
    \end{align*}
    Since $j$ is injective, we have the identity $\lambda_B^\omega  a  \lambda_{B^{-1}}^\omega = \lambda_{t(B)}^\omega  a$. 
    Furthermore, since $j(a)$ is supported on $\Iso(\cG)$, it follows that for all $E \subseteq \cG^{(0)}$, we have the identity $\lambda_E^\omega a = a \lambda_E^\omega$. Therefore,
    \begin{align*}
        \lambda_B^\omega a &= \lambda_B^\omega \lambda_{s(B)}^\omega a = \lambda_B^\omega a \lambda_{s(B)}^\omega \\
        &= \lambda_{t(B)}^\omega a \lambda_B^\omega = a \lambda_{t(B)}^\omega \lambda_B^\omega \\
        &= a \lambda_B^\omega \;. 
    \end{align*}
    Since the elements of the form $\lambda_B^\omega$ are SOT-dense in $L_\omega(\cG)$, we get the converse. 
\end{proof}

\subsection{Kleppner's condition and Factoriality}
\label{Section: Kleppner's condition and Factoriality}
Motivated by Lemma \ref{lemma:j(a) con inv twisted}, we say a conjugation invariant Borel set $E \subset \Iso(\cG)$ is \emph{central} if there exists a function $f: E \to \C$ such that 
\begin{equation}\label{eq:f central set}
    f(ghg\inv) =  \overline{\omega(ghg\inv, g)}\omega(g, h) f(h)
\end{equation}
for almost all $h \in E$ and all $g \in s\inv(s(h))$.

\begin{theorem}\label{theorem:Factoriality twisted}
    Let $(\cG, \mu, \omega)$ be a twisted measured discrete groupoid. The following two conditions are equivalent:
    \begin{enumerate}
        \item Any central Borel set $E\subset \Iso(\cG)$ with $\mu_s(E) < \infty$ is contained in $\cG\zero$.
        \item The center of $L_\omega(\cG)$ equals $\cZ \left(L_\omega(\cG)\right) = L^\infty(\cG\zero, \mu)^\cG$.
    \end{enumerate}    
\end{theorem}
\begin{proof}
    Suppose first that $E \subset \Iso(\cG)$ is central Borel set with $\mu_s(E) < \infty$ which is not contained in $\cG\zero$. Let $f: E \to \C$ be a function satisfying \eqref{eq:f central set}. Then a calculation shows that $f\cdot 1_E$ defines an element in $\cZ(L_\omega(\cG))$, which is moreover not contained in $L^\infty(\cG\zero,\mu)$.

    Conversely, suppose $a$ is an element in $\cZ(L_\omega(\cG))$ which is not contained in $L^\infty(\cG\zero,\mu)$. By subtracting its conditional expectation, we can assume that $\supp(j(a)) \cap \cG\zero = \emptyset$. As in the proof of Theorem \ref{theorem:Factoriality untwisted} we find a conjugation invariant set $E \subset \supp(j(a))$ with $\mu_s(E) < \infty$. By Lemma \ref{lemma:j(a) con inv twisted}, $E \subset \Iso(\cG)$ is a central set with the function $j(a)|_E$, which is not contained in $\cG\zero$.
\end{proof}

\begin{corollary}\label{corollary:factoriality twisted}
     Let $(\cG, \mu, \omega)$ be a twisted measured discrete groupoid. The following two conditions are equivalent.
    \begin{enumerate}
        \item The groupoid $\cG$ is ergodic and any central Borel set $E \subset \Iso(\cG)$ with $\mu_s(E) < \infty$ is contained in $\cG\zero$.
        \item The twisted groupoid von Neumann algebra $L_\omega(\cG)$ is a factor.
    \end{enumerate}
\end{corollary}


Inspired by \cite{Kle62} we introduce the following definition.

\begin{definition}
    Let $(\cG, \omega)$ be a twisted Borel groupoid. We say that a Borel bisection $A$ of $\Iso(G)$ is \emph{$\omega$-regular} if $\omega(y,g)= \omega(g,x)$ for all $g \in \cG$ and $x, y \in A$ satisfying $gxg\inv = y$. We say that a twisted measured discrete groupoid $(\cG, \mu, \omega)$ satisfies \emph{Kleppner's condition} if whenever $A \subset \cG$ is a $\omega$-regular bisection with finite measure conjugacy class, then $A \subset \cG\zero$.
\end{definition}

We now show that Kleppner's condition is necessary for such a twisted groupoid von Neumann algebra to be a factor. It is unclear if the converse to the following lemma is true for ergodic twisted groupoids. 

\begin{lemma}\label{lemma:Kleppner's condition is needed}
    Let $(\cG, \mu, \omega)$ be a twisted measured discrete groupoid. If $L_\omega(\cG)$ is a factor, then $\cG$ satisfies Kleppner's condition.
\end{lemma}
\begin{proof}
    Suppose that $\cG$ does not satisfy Kleppner's condition. Let $A$ be an $\omega$-regular bisection with $\mu_s(\Omega_A) < \infty$ and $A \not\subset \cG\zero$. Note that the latter condition is equivalent to $\Omega_A \not\subset \cG^{(0)}$. Let $\cA := \{ (g, a) \in \cG \times A\ |\ s(g) = t(a)\}$  and define
    \begin{equation*}
        f_0: \cA \to \T : f_0(g,a) = \omega(g, a)\omega(ga, g\inv).
    \end{equation*}
    Whenever $x,y \in A$ and $g,h \in \cG$ are such that $gxg\inv = hyh\inv$, note that the $\omega$-regularity of $A$ implies that $\omega(y, h\inv g) = \omega(h\inv g, x)$. The following direct computation using this observation then shows that $f_0(g,x) = f_0(h,y)$. We indicate in red in each line which terms are changing to the next line.
    \begin{align*}
        \omega(g, x) \omega(gx, g\inv) \overline{\omega(h, y)}\ {\color{red} \overline{\omega(hy, h\inv)}} &= \omega(g, x) {\color{red} \omega(gx, g\inv)} \overline{\omega(h, y)} {\color{red} \omega(gxg\inv, h)}\\
        &= {\color{red} \omega(g,x) \omega(gx, g\inv h)} \omega(g\inv, h)\overline{\omega(h, y)}\\
        &= {\color{red} \omega(g, xg\inv h) \omega(x, g\inv h)} \omega(g\inv, h) \overline{\omega(h, y)}\\
        &= {\color{red} \omega(g, g\inv h y) \omega(g\inv h, y)} \omega(g\inv, h) \overline{\omega(h, y)}\\
        &= \omega(g, g\inv h) \omega(h, y) \omega(g\inv, h) \overline{\omega(h, y)}\\
        &= 1\,.
    \end{align*}
    
    The map $(g,a) \mapsto gag\inv$ is a countable-to-one Borel surjection from $\cA$ onto $\Omega_A$. Therefore there exists a Borel section $\sigma: \Omega_A \to \cA$. We then get that the formula $f(x) = f_0(\sigma(x))$ defines a Borel map which is independent of the choice of section. Additionally, another direct computation below shows that $f$ satisfies
    \begin{equation*}
        f(gxg\inv) = \omega(g, x)\omega(gx, g\inv) f(x)
    \end{equation*}
    for all $x \in \Omega_A$ and $g \in \cG$. Indeed, writing $x = hah\inv$ for some $a \in A$, we make the following computation, where we again indicate in red in each line which terms are changing to the next line,
    \begin{align*}
        \omega(hg, a) &{\color{red} \omega(hga, g\inv h\inv)} \overline{\omega(g,a)}\ \overline{\omega(ga, g\inv)}\ \overline{\omega(h, gag\inv)}\ {\color{red}\overline{\omega(hgag\inv, h\inv)}}\\
        &= {\color{red} \omega(hg, a) \omega(hga, g\inv)} \overline{\omega(g\inv, h\inv)}\ \overline{\omega(g,a)}\ \overline{\omega(ga, g\inv)}\ \overline{\omega(h, gag\inv)}\\
        &= {\color{red}\omega(hg, ag\inv)} \omega(a, g\inv) {\color{red}\overline{\omega(g\inv, h\inv)}}\ \overline{\omega(g,a)}\ \overline{\omega(ga, g\inv)}\ \overline{\omega(h, gag\inv)}\\
        &= {\color{red}\omega(h, gag\inv)} \omega(g, ag\inv)  \omega(a, g\inv) \overline{\omega(g,a)}\ \overline{\omega(ga, g\inv)}\ {\color{red}\overline{\omega(h, gag\inv)}}\\
        &= 1\,.
    \end{align*}
    
    By construction $f \lambda_{\Omega_A} \in L_\omega(\cG) \setminus L^\infty(\cG\zero, \mu)$. We now prove that $f \lambda_{\Omega_A} \in \cZ(L_\omega(\cG))$. Let $B \subset \cG$ be any Borel bisection. For any $b \in B$ and $gag^{-1} \in \Omega_A$, such that $t(g) = s(b)$, we have 
    \begin{align*}
        b(gag^{-1}) = bgag^{-1}s(b) = (bg)a(bg)^{-1}b \in \Omega_A \cdot B\;.
    \end{align*}
    Therefore $B \cdot \Omega_A \subseteq \Omega_A \cdot B$. Similarly, $\Omega_A \cdot B \subseteq B \cdot \Omega_A$ and thus $B \cdot \Omega_A = \Omega_A \cdot B$. For any element $bx = y b \in B \cdot \Omega_A$ with $x, y \in \Omega_A$ we have
    \begin{equation*}
        j(\lambda_B (f \lambda_{\Omega_A}))(bx) = \omega(b, x)f(x) = \omega(y, b) f(y) = j\left( (f \lambda_{\Omega_A})\lambda_B \right)(yb)\,,
    \end{equation*}
    which shows that $f \lambda_{\Omega_A} \in \cZ(L_\omega(\cG))$.
\end{proof}

\section{Corollaries and examples}\label{Section:Corollaries and examples}

\subsection{Group bundles}

When $(\cG,\mu)$ is a countable discrete measured groupoid we write $\Gamma_x=\{g\in \cG: s(g)=t(g)=x\}$ for the \emph{isotropy group} at $x\in \cG\zero$. In the special case that $\cG=\Iso(\cG)$ we get that $\cG$ is a measurable field of groups $\cG=\bigsqcup_{x\in \cG\zero} \Gamma_x$ and the characterization from Theorem \ref{theorem:Factoriality untwisted} reduces to the following. 

\begin{proposition}
    \label{proposition: group bundles}
    Suppose that $(\cG,\mu)$ is a countable discrete groupoid such that $\cG=\Iso(\cG)$. Then $\cG$ is a measurable field of countable discrete groups $\cG=\bigsqcup_{x\in \cG\zero }\Gamma_x$. We have that $\mathcal{Z}(L(\cG))=L^{\infty}(\cG\zero,\mu)$ if and only if $\Gamma_x$ is icc for a.e.\ $x\in \cG\zero$. 
\end{proposition}

\begin{proof}
      Write $Y=\{x\in \cG\zero: \Gamma_x \text{ is not icc}\}$, which is Borel by Lemma \ref{lem:measurable section} below. For an element $g\in \cG$ we write $C_g$ for the conjugacy class of $g$ in $\Gamma_{s(g)}$. 
      We will prove that $\mu(Y)=0$ if and only if for every nontrivial bisection $A\subset \cG\setminus \cG\zero$ we have that $|C_{g}|=+\infty$ for a.e.\ $g\in A$. Then the result follows from Proposition \ref{prop:measure class icc} and Theorem \ref{theorem:Factoriality untwisted}.

     First assume that $\mu(Y)>0$. Denote by $e_x$ the neutral element in $\Gamma_x$ for every $x$ in $X$. By Lemma \ref{lem:measurable section} below there exists a section $\sigma\colon \cG\zero\rightarrow \cG$ such that for every $x\in Y$ we have that $\sigma(x)\neq e_x$ and $\sigma(x)$ has finite conjugacy class. Then $\sigma(Y)\subset \cG\setminus\cG\zero$ is a nontrivial bisection and $|C_g|<+\infty$ for every $g\in \sigma(Y)$ by construction.

     Conversely, assume that $\mu(Y)=0$ and let $A\subset \cG\setminus \cG\zero$ be a nontrivial bisection. By definition $|C_g|=+\infty$ for every $g\in A\cap s^{-1}(\cG\zero\setminus Y)$, which is a conull set inside $A$.
\end{proof}

\begin{lemma}\label{lem:measurable section}
    Let $(\cG,\mu)$ be a discrete Borel groupoid. Write $\Gamma_x=\{g\in \cG:s(g)=t(g)\}$ for $x\in\cG\zero$ and for an element $g\in \Gamma_x$ write $C_g\subset \Gamma_x$ for its conjugacy class. The sets $B=\{x\in \cG\zero \mid \Gamma_x \text{ is trivial}\}$ and $A=\{x\in\cG\zero\mid \Gamma_x \text{ is icc}\}$ are Borel and there exists a Borel section $\sigma\colon \cG\zero\rightarrow  \Iso(\cG)$ such that $\sigma(x)$ is nontrivial and has finite conjugacy class for every $x\in \cG\zero\setminus A$.
\end{lemma}

\begin{proof}
    The source map $s\colon \Iso(\cG)\rightarrow \cG\zero$ is a surjective and countable-to-one Borel map. So by the Lusin--Novikov theorem there exists a countable collection of Borel subsets $U_n \subset \Iso(\cG)$ such that $\Iso(\cG)=\bigcup_n U_n$ and such that $s\colon U_n\rightarrow \cG\zero$ is a Borel isomorphism for every $n$. Note that $x \in B$ if and only if $x \in U_n$ for every $n$, where we identify $x$ and $e_x$. Then we have that $B=\bigcap_{n}(\cG\zero\cap U_n)$, which is Borel. 

    The restriction of $s$ to $\Iso(\cG)\setminus\cG\zero$ is again countable-to-one, so we find countably many Borel section $\tau_n\colon \cG\zero\setminus B\rightarrow \Iso(\cG)$ such that $\Iso(\cG)\setminus \cG\zero=\bigcup_n \tau_n(\cG\zero)$. For $g\in \Gamma_x$ write $C_g\subset \Gamma_x$ for its conjugacy class. If $\sigma\colon \cG\zero\rightarrow \Iso(\cG)$ is a Borel section, then the subset $\bigsqcup_{x\in \cG\zero} C_{\sigma(x)}=\bigcup_n U_n\sigma(\cG\zero)U_n^{-1}\subset \Iso(\cG)$ is Borel. Then, since for any countable-to-one Borel map $q$, the assignment $x\mapsto |q^{-1}(x)|$ is Borel, we get that $x\mapsto |C_{\tau_n(x)}|$ is Borel for every $n$. Now define Borel maps
    \begin{align}
        f_n\colon \cG\zero\setminus B\rightarrow \mathbb{N}\cup \{+\infty\}:\quad f_n(x)=\begin{cases}+\infty &\text{ if }  |C_{\tau_n(x)}|=+\infty,\\
        n &\text{ if }|C_{\tau_n(x)}|<+\infty,\end{cases}
    \end{align} 
    and set $f(x)=\inf_n f_n(x)$. Then we have that $A=f^{-1}(\{+\infty\})\cup B$. Define the Borel section $\sigma\colon \cG\zero\rightarrow \Iso(\cG)$ by
    \begin{align}
        \sigma(x)=\begin{cases}\tau_n(x) &\text{ if } f(x)=n<+\infty,\\
        e_x &\text{ if } x \in A.
        \end{cases}
    \end{align}
    Then $\sigma(x)$ is nontrivial and has finite conjugacy class for every $x\in \cG\zero\setminus A$.
\end{proof}

\subsection{Transformation groupoids}

As mentioned in the introduction Vaes provided necessary and sufficient conditions for when the crossed $L^{\infty}(X, \mu) \rtimes G$ by a nonsingular group action $G \actson (X,\mu)$ is a factor in \cite{362329}. In this subsection we deduce this characterization as a corollary of Theorem \ref{theorem:Factoriality untwisted}, see Corollary \ref{cor:crossed product} below. In particular, \cite[Corollary]{362329} says that if the action $G\actson (X,\mu)$ is moreover pmp, the crossed product $L^{\infty}(X, \mu)\rtimes G$ is a factor if and only if $G\actson (X,\mu)$ is ergodic and every $g \in G$ satisfying $\mu(\{x\in X: g\cdot x=x\})>0$ has an infinite conjugacy class in $G$.

Recall that to a nonsingular action $G\actson (X,\mu)$ of a countable group $G$ on a standard probability space $(X,\mu)$ we associate the \emph{transformation groupoid} $G\ltimes X$. As a set, we have that $G\ltimes X=G\times X$. The set of composable pairs is given by $(G\ltimes X)^{(2)}=\{((g,h\cdot x),(h, x))\mid g,h\in G, x\in X\}$ and multiplication is defined by $(g,h\cdot x)(h,x)=(gh, x)$. Identifying the unit space $(G\ltimes X)\zero$ with $X$, we get that $s(g,x)=x$ and $t(g,x)=g\cdot x$. There is a canonical isomorphism $L^{\infty}(X, \mu)\rtimes G\cong L(G\ltimes X)$ given by $f \mapsto f$ for $f \in L^\infty(X,\mu)$ and $u_g \mapsto \lambda_{\{g\} \times X}$ for $g \in G$, where the $u_g$ denote the canonical unitaries generating $L(G)$ inside $L^\infty(X,\mu) \rtimes G$.

Recall that for a discrete countable group $G$ a nonsingular action $G \actson (X,\mu)$ is said to be induced from $H \actson Y$ for a subgroup $H < G$ and a Borel subset $Y \subset X$ if $Y$ is $H$-invariant and the sets $\{gY \; | \; g \in G/H \}$ partition $X$ up to a set of measure zero. The proof of the next proposition closely resembles the proof of \cite[Theorem]{362329}.

\begin{proposition}\label{prop:transformation groupoid}
    Let $G \actson (X,\mu)$ be an ergodic nonsingular group action and let $\cG$ be the associated transformation groupoid. Then the following are equivalent.
    \begin{enumerate}
        \item Whenever $G\actson X$ is induced from $H\actson Y$ and $h\in H \setminus\{e\}$ acts trivially on $Y$, $h$ has an infinite conjugacy class in $H$.
        \item The transformation groupoid $\cG$ is icc. 
    \end{enumerate}
\end{proposition}
\begin{proof}
    First assume that $\cG$ is not icc. Then there exists a nontrivial bisection $B\subset \Iso(\cG)\setminus \cG\zero$ which has $\mu$-finite conjugacy class $\Omega_B$. Denote by $\mathcal{P}(G)$ the power set of $G$ and define the function $F\colon X\rightarrow \mathcal{P}(G)$ by 
    \begin{align}
        F(x)=\{g\in G\mid 1_{\Omega_B}(g,x)=1\}\,.
    \end{align}
    Write $Z=\{x\in X \mid F(x)\neq \emptyset\}$. Then $\mu(Z)>0$ since $\mu(s(B))>0$. Furthermore, since $\Omega_B$ is conjugation-invariant,  we have that $F(g\cdot x)=gF(x)g\inv$ for every $x\in X$ and $g\in G$. In particular, $Z$ is $G$-invariant and we get that $\mu(Z)=1$. 
    Also we have that $F(x)\subset G\setminus\{e\}$ for every $x\in Z$ because $B\cap \cG\zero =\emptyset$. Since $\mu_s(\Omega_B)=\int_{X}|F(x)|\mathrm{d}\mu(x)$, we have that $|F(x)|<+\infty$ for a.e.\ $x\in Z$. There are only countably many nonempty finite subsets $C\subset G\setminus \{e\}$, so we can pick a finite nonempty set $C\subset G\setminus\{e\}$ such that the set $Y=\{x\in X:F(x)=C\}$ has positive measure. 
Then we have that $\mu(gY\cap Y)>0$ if and only if $gCg\inv= C$ if and only if $gY=Y$. So, if we write $H=\{h\in G: hCh\inv= C\}$ we get that $G\actson X$ is induced from $H\actson Y$. When $g\in C$ and $y\in Y$, we have that $1_{\Omega_B}(g,y)=1$. This means that $(g,y)\in \Omega_B\subset \Iso(\cG)$, so that $g\cdot y=y$. In particular we have that $F(g\cdot y)=F(y)$ and it follows that  $g\in H$. So $C$ is a finite subset of $H\setminus \{e\}$ that is invariant under conjugation by $H$ and all its elements act trivially on $Y$, contradicting statement $1$ of the proposition. 

Conversely, assume that $G\actson X$ is induced from $H\actson Y$ and that $h\in H\setminus\{e\}$ acts trivially on $Y$ such that $h$ has finite conjugacy class $C$. Then we have that
\begin{align}
    \Omega_{\{h\}\times Y}=\bigcup_{g\in G}gCg\inv\times gY\,.
\end{align}
Since $gCg\inv =kCk\inv$ whenever $gH=kH$, we get that 
\begin{align}
    \mu_s( \Omega_{\{h\}\times Y})=\sum_{g\in G/H}\mu_s(gCg\inv\times gY)= |C|<+\infty.
\end{align}
So $\{h\}\times Y$ is a partial bisection that has $\mu$-finite conjugacy class, contradicting statement $2$ of the proposition. 
\end{proof}

Using that $L(G\ltimes X)\cong L^\infty(X, \mu)\rtimes G$ and that $L^{\infty}(X, \mu)^G\subset \mathcal{Z}(L^\infty(X, \mu)\rtimes G)$ we thus obtain the following as a corollary of Theorem \ref{theorem:Factoriality untwisted} and Proposition \ref{prop:transformation groupoid}. 

\begin{corollary}[{\cite[Theorem]{362329}}]\label{cor:crossed product}
    Let $G\actson (X,\mu)$ be a nonsingular group action. Then the following are equivalent.
    \begin{enumerate}
        \item The action $G\actson X$ is ergodic and whenever $G\actson X$ is induced from $H\actson Y$ and $h\in H\setminus\{e\}$ acts nontrivially on $Y$, $h$ has an infinite conjugacy class in $H$.
        \item The crossed product $L^\infty(X, \mu)\rtimes G$ is a factor.
    \end{enumerate}
\end{corollary}

\subsection{Full subgroupoids and partial actions}

In the previous section we have demonstrated that factoriality of crossed products $L^\infty(X, \mu) \rtimes G$ is inherently linked to factoriality of certain subgroupoids of $G \ltimes X$. First we describe subgroupoids which share the same ideal structure as our original groupoid. 

\begin{definition}
    Let $\cG$ be a countable Borel groupoid. Given any Borel subset $K \subset \cG^{(0)}$, the induced subgroupoid $\cG|_K$ is defined as the product $ \cG|_K := K \cdot \cG \cdot K$. A Borel subset $K \subset \cG^{(0)}$ is said to be $\cG$-\emph{full} if $t(\cG \cdot K) = \cG^{(0)}$. If $\cG$ is furthermore endowed with a measure $\mu$, we say that $K \subset \cG^{(0)}$ is $\mu$-\emph{full} if $\mu(\cG^{(0)} \setminus t(\cG \cdot K)) = 0$. 
\end{definition}

As an example, when $\cG$ is an equivalence relation a subset $K \subset \cG\zero$ is full precisely when it intersects every orbit.

For what follows, we let $\cR$ denote the full equivalence relation on the set $\mathbb{N}$ of positive integers. Our goal is to show that subgroupoids induced from full subsets of $\cG^{(0)}$ share the same ideal structure as $\cG$. What follows is essentially the same proof as in \cite[Lemmas 2.2, 2.3, 2.4, Proposition 2.5]{Carlsen17} in the ample groupoid case. 

\begin{lemma}
\label{Lemma: fullness inductive step}
    Let $K \subset \cG^{(0)}$ be $\cG$-full. There is a bisection $W \subset \cG \times \cR$ with $s(W) \subset K \times \mathbb{N}$ and $t(W) = \cG^{(0)} \times \{1\}$. If $K$ is $\mu$-full, then the same conclusion holds almost everywhere. 
\end{lemma}
\begin{proof}
    First we construct a sequence $B_n$ of bisections of $\cG^{(0)}$ with mutually disjoint targets such that $\bigcup_n t(B_n) = \cG^{(0)}$. Using Lusin--Novikov, there is a collection $V_n \subset \cG \cdot K$ of bisections which partition $\cG \cdot K$. These bisections need not have mutually disjoint target map, but we can define $B_n$ as follows:
    \begin{enumerate}
        \item $B_1 = V_1$. 
        \item Define 
            \begin{align*}
            B_n = V_n \setminus \left(\bigcup_{k < n} t^{-1}(t(V_k)) \right)\;.
    \end{align*}
    \end{enumerate}
    
The bisection $W := \bigcup_n B_n \times \{(1,n)\}$ gives us the correct source and target. 
\end{proof}
\begin{proposition}
    Let $K \subset \cG^{(0)}$ be $\cG$-full. There is a Borel bisection $W \subset \cG \times \cR$ with $s(W) = K \times \mathbb{N}$ and $t(W) = \cG^{(0)} \times \mathbb{N}$. If $K$ is $\mu$-full, then the same conclusion holds almost everywhere. 
\end{proposition}
\begin{proof}
    The proof is a back and forth construction using Lemma~\ref{Lemma: fullness inductive step}, and is identical to the proof in \cite[Lemma 2.4]{Carlsen17}, so we omit the details here. 
\end{proof}
\begin{corollary}
    Let $K \subset \cG^{(0)}$ be $\cG$-full. The groupoid $\cG|_K \times \cR$ is isomorphic to $\cG \times \cR$. 
\end{corollary}
\begin{proof}
    Let $W \subset \cG \times \cR$ be a Borel bisection with $s(W) = K \times \mathbb{N}$ and $t(W) = \cG^{(0)} \times \mathbb{N}$. The map 
    \begin{align*}
        \cG \times \cR \to \cG |_K \times \cR: g \mapsto W^{-1} g W \;,
    \end{align*}
   where $W^{-1}gW = w^{-1}gw$ for the unique $w \in W$ with $t(g) = t(w)$, gives us the groupoid isomorphism, with inverse $g \mapsto WgW^{-1}$. 
\end{proof}
\begin{corollary}
    Let $(\cG,\mu)$ be a discrete measured groupoid. Let $K \subset \cG^{(0)}$ be $\mu$-full. The following are equivalent. 
    \begin{enumerate}
        \item $\cG$ is icc. 
        \item $\cG|_K$ is icc. 
    \end{enumerate}
\end{corollary}
\begin{proof}
    Consider the identity $L(\cG) \ovt B(\ell^2(\mathbb{N})) = L(\cG \times \cR) \cong L(\cG|_K \times \cR) = L(\cG|_K) \ovt B(\ell^2(\mathbb{N}))$. Since the center of the tensor product of von Neumann algebras is the tensor product of their centers, the result now follows from Theorem \ref{theorem:Factoriality untwisted}.
\end{proof}

\begin{proposition}
    Let $G \actson X$ be a Borel action of a countable group $G$ on a standard Borel space $X$ and let $Y \subset X$ be a Borel subset. Assume there is a subgroup $H \subset G$ such that $(G \ltimes X)|_Y = H \ltimes Y$. The space $Y$ is $\cG$-full in $G \ltimes X$ if and only if $G \actson X$ is induced from $H \actson Y$. 

    Furthermore, if $\mu$ is a probability measure on $X$ so that $G\actson (X,\mu)$ is nonsingular, then $Y$ is $\mu$-full in $(G \ltimes X, \mu)$ if and only if $G \actson (X,\mu)$ is induced from $H \actson Y$. 
\end{proposition}

\begin{proof}
    If $G \actson X$ is induced from $H \actson Y$, then $Y$ is full by definition. Conversely, note that by fullness of $Y$, it follows that $\bigcup_{g \in G} gY = t((G \ltimes X)Y) = X$. To show the sets $gY$ are mutually disjoint, let $g \not\in H$. If there is some $y \in gY \cap Y$, then $(g,y) \in (G \ltimes X)|_Y = H \ltimes Y$. Thus, $g \in H$, reaching a contradiction. The measurable setting is the same, with equality replaced with equality up to a negligable subset. 
\end{proof}

Note that subgroupoids of $G\ltimes X$ are not necessarily crossed products. Indeed, even \emph{induced subgroupoids} of $G \ltimes X$ are not necessarily a crossed product. 
\begin{example}
    Let $\mathbb{Z} \actson \mathbb{Z}$ by left translation and consider the subset $\mathbb{N} \subset \mathbb{Z}$ of positive integers. If the groupoid $(\mathbb{Z} \ltimes \mathbb{Z})|_\mathbb{N}$ is equal to $\langle a \rangle \ltimes \mathbb{N}$ for some $a \in \mathbb{Z}$, then $\langle a \rangle$ must contain all the positive integers, which implies that $a =1$, deriving a contradiction. 
\end{example}
To get around this issue, one can consider at partial actions instead. The following is a translation of the partial dynamical system construction in the topological case to the Borel setting. See the book \cite{Exel17} and the references therein for more details. 

By a partial automorphism of a standard Borel space $X$, we mean a function $\sigma: X_d \to X_r$, where $\sigma$ is a Borel bijection between Borel subsets $X_d,X_r \subset X$. A partial action $G \curvearrowright_\sigma X$ is a collection of Borel sets $\{X_g \;|\; g \in G \}$ and partial isomorphisms
    \begin{align*}
        \sigma_g : X_{g^{-1}} \to X_g 
    \end{align*}
    such that $\sigma_e = \id_X$ and for all $g,h \in G$, $\sigma_g \circ \sigma_h \subset \sigma_{gh}$. Here, we are treating $\sigma_g$ and $\sigma_h$ by its graph, and so
    \begin{align*}
        \sigma_g \circ \sigma_h = \{(x,y) \;|\; \text{there exists }z \text{ such that }(x,z) \in \sigma_h, (z,y) \in \sigma_g \}\;. 
    \end{align*}

    Given $G \curvearrowright_\sigma X$, we can define the Borel groupoid 
    \begin{align*}
        G \ltimes_\sigma X = \{(g,x) \;|\; g \in G, x \in X_{g^{-1}}\} \;. 
    \end{align*}
    The source, target, multiplication, and inversion maps are defined exactly as in the case of global group actions. In particular, we have 
    \begin{align*}
        \Iso(G \ltimes_\sigma X) = \{(g,x) \;|\; x \in X_g \cap X_{g^{-1}}, \sigma_g(x) = x \} \;. 
    \end{align*}
    
    For $g \in G$, we let $\Fix(g) := \{x \in X_g \cap X_{g^{-1}} \;|\; \sigma_g(x) = x \}$. We have the identity 
    \begin{align*}
        \Omega_{\{g\} \times \Fix(g)} = \bigcup_{h \in G} \{hgh^{-1}\} \times (X_h \cap \Fix(g))  \;. 
    \end{align*}

The situation for partial crossed products is quite different in this sense. 
\begin{proposition}
    Let $G \actson_\sigma X$ be a partial action on a standard Borel space $X$. For any Borel subset $Y \subset X$, the restriction map induces a partial action $G \actson_\sigma Y$ such that the induced subgroupoid $(G \ltimes X)|_Y := \{(g,x) \in G \ltimes X \; | \; x, g\cdot x \in Y\}$ is equal to $G \ltimes Y$. 
\end{proposition}
\begin{proof}
    Set $Y_g = \{x \in X_g \; | \; x \in Y, \sigma_{g^{-1}}(x) \in Y \}$. For any $x \in Y_g$, $\sigma_{g^{-1}}(x) \in Y_{g^{-1}}$.  It follows that $\sigma_g\colon Y_{g\inv} \rightarrow Y_{g}$ is a bijection with inverse $\sigma_{g^{-1}}$. The only non-trivial fact to verify is that $(\sigma_g \circ \sigma_h)(x) \in Y_{gh}$ if $x \in \sigma_h^{-1}(Y_h \cap Y_{g\inv})$. Since $(\sigma_g \circ \sigma_h)(x) = \sigma_{gh}(x) \in Y \cap X_{gh}$, it is enough to verify that $x \in Y_{(gh)^{-1}}$. This follows from the fact that $x, \sigma_{gh}(x) \in Y$. Finally, to see that $(G \ltimes X)|_Y = G \ltimes Y$, observe that $\{g\} \times Y_{g^{-1}}$ form a basis for both groupoids. 
\end{proof}

Indeed, partial actions always come from the restriction of a global action. The proof is a modification of \cite[Theorem 2.5]{Abadie03} in the topological case. 
\begin{proposition}
    Let $G \actson_\sigma Y$ be a partial action on a standard Borel space $Y$. Then there is an embedding $Y \hookrightarrow X$ into a standard Borel space $X$ and a global action $G \actson_\theta X$ such that $G \ltimes Y = (G \ltimes X)|_Y$. Furthermore, $X$ can be chosen such that $Y$ is $G \ltimes X$-full. 
\end{proposition}

\begin{proof}
    On the Borel space $G \times Y$, define a Borel relation $R \subset (G \times Y)^2$ as follows: for two points $(g,x),(h,y) \in G \times Y$, say $(g,x)R(h,y)$ if $x \in Y_{g^{-1}h}$ and $\sigma_{h^{-1}g}(x) = y$. First let us demonstrate that $R$ is Borel.  By permuting the coordinates, assume that $R \subset G^2 \times Y^2$. We treat $\sigma_{h^{-1}g} \subset Y \times Y$ as its graph. We have the identity 
    \begin{align*}
        R = \bigsqcup_{g,h \in G} \{(g,h)\} \times \sigma_{h^{-1}g}\;.
    \end{align*}
    Furthermore, for any $(g,x) \in G \times Y$, its equivalence class is countable. Thus, $R$ is a countable Borel equivalence relation on the Borel space $G \times Y$. Next, we show that $R$ is a smooth equivalence relation. Let   $[(g,x)]$ denote the equivalence class of an element $(g,x) \in G \times Y$ under $R$. Let $(g_n)_n$ be an enumeration of $G$ and define
    \begin{align*}
        A=\bigcup_{n=0}^\infty\left(\{g_n\}\times \left(Y\setminus \bigcup_{k=0}^{n-1}Y_{g_n^{-1}g_k}\right)\right)\,.
    \end{align*}
   We claim that $A$ intersects each orbit exactly once. Suppose that $(g_i,x),(g_j,y)\in A$ are such that $(g_i,x)R (g_j,y)$. Then, if $i\neq j$, we may assume that $i<j$. In that case we get that $y\in Y_{g_j\inv g_i}$, by definition of $R$. At the same time, by definition of $A$, we get that $y\notin Y_{g_j\inv g_i}$. Therefore we must have that $i=j$ and it follows that also $x=y$. 

   Next, we show by induction that $[(g,x)]\cap A$ is nonempty for every $(g,x)\in G\times Y$. Take an arbitrary element $(g_i,x)\in G\times Y$. If $i=0$ we are done. Now assume that $i>0$. If $x\notin \cup_{k=0}^{i-1}Y_{g_i^{-1}g_k}$, then we have that $(g_i,x)\in A$ by definition. If $x\in \cup_{k=0}^{i-1}Y_{g_i^{-1}g_k}$, there exists a $k<i$ such that $x\in Y_{g_i\inv g_k}$ and we get that $(g_i,x) R (g_k,\sigma_{g_i\inv g_k}(x))$. By the induction hypothesis we get that $[(g_i,x)]\cap A=[(g_k,\sigma_{g_i\inv g_k}(x))]\cap A\neq \emptyset$.

   So $A$ is a fundamental domain for the equivalence relation $R$ and the quotient $X = (G \times Y)/R$ is a well-defined standard Borel space. 

    There is a Borel embedding $\iota: Y \hookrightarrow X : x \mapsto [e,x]$ and a $G$-action on $X$ defined by $\theta_g([h,x]) = [gh,x]$. By our construction of the equivalence relation, for any $x \in Y_{g^{-1}}$, $\theta_g([e,x]) = [g,x] = [e,\sigma_g(x)]$. Thus, we have $G \ltimes Y = (G \ltimes X)|_Y$. Finally, note that for all $(g,x) \in G \times Y$, we have $[(g,x)] = g \cdot \iota(x)$, and thus $Y$ is $G \ltimes X$-full. 
\end{proof}
\begin{corollary}
    Let $G \actson Y$ be a partial action and let $\mathcal{R}$ be the full equivalence relation on $\mathbb{N}$. There is a global action $G \actson X$ such that $(G \ltimes Y) \times \mathcal{R} \cong (G \ltimes X) \times \mathcal{R}$. 
\end{corollary}

Using the notion of a full subset, we can now prove an analogue of the characterization of factoriality for transformation groupoids for partial actions. 
\begin{definition}
    Let $G \actson X$ be a partial action. 
    \begin{enumerate} 
        \item We say that a Borel subset $Y \subset X$ is \emph{$G$-invariant} if for all $g \in G$, $g(Y \cap X_{g^{-1}}) = Y \cap X_g$. 
        \item We say that an element $g \in G$ \emph{acts trivially on $X$} if $g\cdot x = x$ for all $x \in X_{g^{-1}}$. In particular, $X_{g^{-1}} = X_g$. 
        \item Let $H \leq G$ and let $Y \subset X$ be an $H$-invariant Borel subset. We say that $G \actson X$ is induced from $H \actson Y$ if $Y$ is $\cG$-full in $G \ltimes X$ and $G \ltimes Y = H \ltimes Y$. 
        \item Let $\mu$ be a finite measure on $X$. Let $H \leq G$ and let $Y \subset X$ be an $H$-invariant Borel subset. We say that $G \actson X$ is induced from $H \actson Y$ if $Y$ is $\mu$-full in $G \ltimes X$ and $G \ltimes Y = H \ltimes Y$. 
    \end{enumerate} 
\end{definition}

\begin{theorem}\label{Theorem: partial action}
    Let $G \actson (X,\mu)$ be a non-singular ergodic partial action. The following are equivalent:
    \begin{enumerate}
    \item If $G \actson (X,\mu)$ is induced from $H \actson Y$ and $h \in H \setminus \{e\}$ is such that $Y_{h^{-1}}$ has non-zero measure and $h$ acts trivially on $Y$, then $h$ has infinite $H$-conjugacy class. 
    \item $G \ltimes X$ is icc. 
    \end{enumerate}
\end{theorem}

\begin{proof}
    First assume that $G \ltimes X$ admits a conjugation invariant subset $B \subset \Iso(G \ltimes X) \setminus (G \ltimes X)^{(0)}$ with finite measure. As before, we set 
    \begin{align*}
        F(x) = \{g \in G \;|\; (g,x) \in B \} 
    \end{align*}
    and observe that $F(gx) = gF(x)g^{-1}$ if $x \in X_{g^{-1}}$. As in the global case, we may fix $C \subset G \setminus \{e\}$ finite for which the set 
    \begin{align*}
        Y = \{x \in X \;|\; F(x) = C \} 
    \end{align*}
    has positive measure. By ergodicity, $Y$ is $\mu$-full. For any $g \in G$, 
    \begin{align*}
        g \cdot (Y \cap X_{g^{-1}}) = \{x \;|\; F(x) = gCg^{-1}\} \cap X_g \;. 
    \end{align*}
    In particular, if $gCg^{-1} \neq C$, then $g(Y \cap X_{g^{-1}}) \cap Y = \emptyset$. Since $Y_g := g(Y \cap X_{g^{-1}}) \cap Y$, it follows that $Y_g = \emptyset$ unless $gCg^{-1} = C$. On the other hand, if $gCg^{-1} = C$, we get the identity $Y_g = Y \cap X_g$. Set $H = \{g \in G: gCg^{-1} = C\}$. We therefore have the identity $H \ltimes Y = G \ltimes Y$. Furthermore, for any $h \in C$ and $y \in Y_{h^{-1}}$, since $F(y) = C$, we have that $(h,y) \in B$. That is, $h \cdot y = y$. Thus, $h$ is an element that acts trivially on $Y$ with finite conjugacy class. 

    Conversely, if there are $H \leq G$ and $Y \subset X$ with $H \ltimes Y = G \ltimes Y$ and an element $h \in H$ that acts trivially on $Y$ and has finite $H$-conjugacy class, it suffices to demonstrate that $B:= (Y_h \cap Y_{h^{-1}}) \times \{h\}$ has $\mu$-finite conjugacy class in $H \ltimes Y$. This follows from the fact that
    \begin{align*}
        \Omega_B \subset \bigcup_{g \in H} \{ghg^{-1}\} \times Y\;
    \end{align*}
    and that there are only finitely many conjugates on the right hand side. 
\end{proof}
\begin{remark}
    Due to the generality of induced subsystems in the partial dynamical case, note that condition 1 in Theorem~\ref{Theorem: partial action} is distinct from the criterion given in Proposition~\ref{prop:transformation groupoid}, even if the original action $G \actson X$ is global. 
\end{remark}
By the Nielsen--Schrier Theorem, all subgroups of free groups are free. Thus, we have the following Corollary: 
\begin{corollary}
    Let $\F$ be a free group and suppose that $\F \actson (X,\mu)$ is an ergodic nonsingular partial action. The following are equivalent: 
    \begin{enumerate}
        \item The groupoid $(\F \ltimes X,\mu)$ is not icc. 
        \item There is a non-trivial $g \in \mathbb{F}$ and a $\langle g \rangle$-invariant subset $Y \subset (X,\mu)$ such that $\F \actson X$ is induced from $\langle g \rangle \actson Y$ and $\langle g \rangle \actson Y$ is not essentially free.
    \end{enumerate}
In particular, for nonsingular partial $\Z$-actions, the groupoid $(\Z \ltimes X, \mu)$ is icc if and only if $\Z \actson (X,\mu)$ is essentially free. 
\end{corollary}

\begin{proof}
    Assume first that $\mathbb{F} \ltimes X$ is not icc. This means that there is some subgroup $H \leq \mathbb{F}$ and some $H$-invariant subspace $Y$ that induces $\mathbb{F} \actson (X,\mu)$ for which for some non-trivial $h \in H$, $h$ acts trivially on $Y$ but $h$ has finite conjugacy class. By the Nielsen--Schrier Theorem, $H$ must be a free group, and since it admits a non-trivial element with finite conjugacy class, $H$ must be cyclic. Let $g \in H$ be the generator for $H$. The group $H$ cannot act essentially free since for some $n \in \mathbb{Z}$, $h = g^n$ acts trivially on $Y$. 

    Conversely, assume that there is an induced subsystem $H \actson Y$ for a cyclic group $H = \langle g \rangle \leq \mathbb{F}$. Furthermore, assume that $H \actson Y$ is not essentially free. This means that there is a non-zero $n$ for which the set $E := \{y \in Y_{g^{-n}} \cap Y_{g^n} \; | \; g^n \cdot y = y \}$ has positive measure. The set $B :=\{ g^n \} \times E$ is a bisection in $H \ltimes Y$ with $\Omega_B \subseteq \{g^n\} \times (Y_{g^n} \cap Y_{g^{-n}}) $, and thus $H \ltimes Y$ is not icc. Since $\mathbb{F} \ltimes X$ is induced from $H \ltimes Y$, it is also not icc. 
\end{proof}

\subsection{Deaconu--Renault Groupoids}

Let $(X,\mu)$ be a Borel probability space and let 
    \begin{align*}
        \sigma : X \to X 
    \end{align*}
    be a measurable map. Define the Deaconu--Renault groupoid of $(X,\sigma,\mu)$ as follows: as measure space
    \begin{align*}
        \cG(\sigma) = \{(x,n-m, y)\;|\; x,y \in X, m,n \in \mathbb{Z}_{\geq 0}, \sigma^n(x) = \sigma^m(y) \} \;,
    \end{align*}
    where the measure is given by the induced measure on $X \times \mathbb{Z} \times X$. The source and target maps are 
    \begin{align*}
        s: (x,k,y) \mapsto y \text{ and } t: (x,k,y) \mapsto x \;. 
    \end{align*}
    These are Borel since they are projection maps. Inversion is then $(x,k,y)^{-1} = (y,-k,x)$, which is Borel since $k \mapsto -k$ is Borel. Finally, multiplication is Borel since it is a composition of projections and addition.

    A computation of the isotropy subgroupoid gives us the identity 
    \begin{align}\label{Equation: B_n}
        \Iso(\cG(\sigma)) = \{(x,k-l,x)\;|\; \sigma^{k}(x) = \sigma^l(x) \} \;. 
    \end{align}
    Furthermore, given any $(x,n,y) \in \cG(\sigma)$ and $(y,m,y) \in \Iso(\cG(\sigma))$, conjugation gives us the identity
    \begin{align*}
        (x,n,y) \cdot (y,m,y) \cdot (y,-n,x) = (x,n+m-n,x) = (x,m,x) \;. 
    \end{align*}
    Fix any $n \in \mathbb{Z}$, let $B_n := (X \times \{n\} \times X) \cap \Iso(\cG(\sigma))$. This is a conjugation invariant Borel bisection by the previous identity. For what follows, we say that a measurable map $\sigma: (X,\mu)\to(X,\mu)$ is essentially free if the set 
    \begin{align*}
        \{x \in X \;|\; \text{ there exists }m,n \geq 0 \text{ such that } \sigma^m(x) = \sigma^n(x) \} 
    \end{align*}
    has measure zero. 
\begin{proposition} Let $\sigma: (X,\mu) \to (X,\mu)$ be a measurable map on a Borel probability space $(X,\mu)$.
    The following are equivalent: 
    \begin{enumerate}
        \item $\cG(\sigma)$ is icc. 
        \item For all $n \in \mathbb{Z} \setminus \{0\}$, we have $\mu_s(B_n) = 0$. 
        \item The map $\sigma$ is essentially free. 
    \end{enumerate}
    \end{proposition}
\begin{proof}
    Note by equation \eqref{Equation: B_n}, $\Iso(\cG(\sigma)) = \bigsqcup_n B_n$. Therefore, by our previous calculations, conditions $1$ and $2$ are equivalent. That $3$ is equivalent to $2$ follows from the fact that 
    \begin{equation*}
        s\left(\bigcup_{n \neq 0} B_n\right) = \{x \;|\;\text{there exists }k > l \geq 0 \text{ such that } \sigma^k(x) \neq \sigma^l(x) \} \;.\qedhere
    \end{equation*}
\end{proof}

\end{document}